\documentclass{amsart}
\usepackage{subfigure}
\usepackage{graphicx,verbatim, amsmath, amssymb, amsfonts, amsthm,mathrsfs, ifthen,mathtools,mathabx,enumerate,epstopdf,placeins,longtable, float, placeins, hyperref, tikz, enumitem, leftidx, tikz-cd}	
\DeclareGraphicsExtensions{.ps}
\newtheorem{proposition}{Proposition}[subsection]
\newtheorem{theorem}[proposition]{Theorem}
\newtheorem{corollary}[proposition]{Corollary}
\newtheorem{lemma}[proposition]{Lemma}
\newtheorem{prop}[proposition]{Proposition}
\newtheorem{cor}[proposition]{Corollary}

\newtheorem{example}[proposition]{Example}

\newtheorem{definition}[proposition]{Definition}

\newtheorem{remark}[proposition]{Remark}

\numberwithin{equation}{section}

 \newcounter{tmp}

\newcounter{margincounter}
\newcommand{\calA}{\mathcal{A}}

\newcommand{\calD}{\mathcal{D}}
\newcommand{\calE}{\mathcal{E}}
\newcommand{\calF}{\mathcal{F}}

\newcommand{\calI}{\mathcal{I}}

\newcommand{\calM}{\mathcal{M}}

\newcommand{\calO}{\mathcal{O}}

\newcommand{\calR}{\mathcal{R}}
\newcommand{\calS}{\mathcal{S}}
\newcommand{\calT}{\mathcal{T}}
\newcommand{\calU}{\mathcal{U}}
\newcommand{\calV}{\mathcal{V}}
\newcommand{\calW}{\mathcal{W}}

\newcommand{\calY}{\mathcal{Y}}

\newcommand{\Kr}{\operatorname{Kr}}

\newcommand{\newword}[1]{\textbf{\emph{#1}}}

\renewcommand{\Join}{\bigvee}
\newcommand{\Meet}{\bigwedge}

\newcommand{\meet}{\wedge}
\newcommand{\join}{\vee}

\newcommand{\F}{\mathcal{F}}

\newcommand{\CJI}{\operatorname{CJI}}
\newcommand{\CMI}{\operatorname{CMI}}

\newcommand{\covers}{{\,\,\,\cdot\!\!\!\! >\,\,}}
\newcommand{\covered}{{\,\,<\!\!\!\!\cdot\,\,\,}}
\newcommand{\1}{\hat{1}}
\newcommand{\0}{\hat{0}}

\newcommand{\module}{\operatorname{mod}}

\newcommand{\Hom}{\operatorname{Hom}}
\newcommand{\End}{\operatorname{End}}
\newcommand{\Gen}{\operatorname{Gen}}
\newcommand{\Cogen}{\operatorname{Cogen}}

\newcommand{\Ext}{\operatorname{Ext}}
\newcommand{\add}{\operatorname{add}}

\newcommand{\filt}{Filt}
\newcommand{\tors}{\operatorname{tors}}
\newcommand{\ftors}{\operatorname{f\textendash tors}}

\newcommand{\Ker}{\operatorname{Ker}}
\newcommand{\cok}{\operatorname{Coker}}
\newcommand{\image}{\operatorname{image}}
\newcommand{\torf}{\operatorname{torf}}

\newcommand{\ind}{\operatorname{ind}}

\newcommand{\ME}{\operatorname{ME}}
\newcommand{\wide}{\operatorname{wide}}
\newcommand{\fwide}{\operatorname{f\textendash wide}}

\newcommand{\NC}{\operatorname{NW}}

\newcommand{\lcm}{\operatorname{lcm}}
\newcommand{\CJR}{\operatorname{CJR}}
\newcommand{\CMR}{\operatorname{CMR}}

\author{Emily Barnard}\address{Department of Mathematical Sciences, DePaul University, Chicago, IL 60614, USA}\email{
e.barnard@depaul.edu}
\author{Gordana Todorov}
\address{Department of Mathematics, Northeastern University, 
Boston, MA 02115, USA}
\email{g.todorov@northeastern.edu}
\author{Shijie Zhu}\address{Department of Mathematics, 
The University of Iowa, 
Iowa City,IA 52242, USA}
\email{shijie-zhu@uiowa.edu}

\title{Dynamical combinatorics and torsion classes}

\begin{document}
\maketitle

\begin{abstract}
For finite semidistributive lattices  
the map $\kappa$ gives a bijection between  the sets of completely join-irreducible elements and completely meet-irreducible elements.

Here we study the $\kappa$-map in the context of torsion classes.
It is well-known that the lattice of torsion classes for an artin algebra is semidistributive, but in general it is far from finite.
We show the $\kappa$-map is well-defined on the set of completely join-irreducible elements, even when the lattice of torsion classes is infinite.
We then extend $\kappa$ to a map on 
torsion classes which have canonical join representations given by the special torsion classes associated to the minimal extending modules
introduced in \cite{BCZ} by the first and third authors and A. Carroll. 

For hereditary algebras, we show that the extended $\kappa$-map on torsion classes is essentially the same as Ringel's $\epsilon$-map on wide subcategories. Also in the hereditary case, we relate the square of $\kappa$ to the Auslander-Reiten translation.
\end{abstract}

\setcounter{tocdepth}{2}
\tableofcontents
\section{Introduction}

In this paper we study the lattice theoretic map $\kappa$ and 
its relations  to  various representation theoretic notions in the lattices of torsion classes over finite dimensional algebras.

Recall that a lattice (or lattice-poset) is a poset in which each pair of elements has a unique smallest upper bound (called the join) and a unique greatest lower bound (called the meet).
We focus on a special class of lattices called semidistributive lattices; they satisfy a semidistributive law which is a weakening of the distributive law (Definition \ref{def}).
Semidistributive lattices arise in algebraic combinatorics as the weak order on a finite Coxeter group, in number theory as the divisibility poset (Example~\ref{number_theory}), and in representation theory as the lattice of torsion classes. 

Fundamental to the study of \emph{finite} semidistributive lattices is a certain function which we denote here as $\kappa$, which maps completely join-irreducible elements to completely meet-irreducible elements (for precise definitions see Definition \ref{join-irreducible}).
Actually for finite lattices, $\kappa$ is a bijection precisely when the lattice is semidistributive  \cite[Corollary~2.55]{FJN}.
In order to define the $\kappa$-map we need the following characterization of completely join-irreducible elements:  $j$ is completely join-irreducible if and only if there is unique element $j_*$ covered by $j$.  
(See Definition~\ref{def:cover} for the definition of cover relations.)
Similarly,  $m$ is completely meet-irreducible if and only if  there is unique element $m_*$ which covers $m$. 
Following \cite{FJN}, we define the $\kappa$-map by using the above characterization of completely join-irreducible elements.

\begin{definition}\label{kappa-map}
Let $j$ be a completely join-irreducible element of a lattice $L$, and let $j_*$ be the unique element covered by $j$.
Define $\kappa(j)$ to be:
$$\kappa(j): = \text{unique max}\{x\in L: j_*\le x\text{ and } j\not\le x\},
\text{ when such an element exists.}$$
\end{definition}

In this paper we concentrate on $\tors \Lambda$, the lattice of torsion classes. Throughout the paper, $\Lambda$ is a finite dimensional algebra; any additional assumptions on the algebra $\Lambda$ will be explicitly stated. A torsion class $\calT$ is a subcategory of $\module \Lambda$ which is closed under isomorphisms, quotients, and extensions. We write $\tors \Lambda$ for the poset whose elements are torsion classes ordered by: $\calT \le \calS$ precisely when $\calT \subseteq \calS$. 
Observe that for two torsion classes $\calT$ and $\calS$ (not necessarily comparable), $\calT\meet \calS = \calT \cap \calS$ and $\calT\join \calS = \filt(\calT\cup \calS)$.
 It is well-known that $\tors \Lambda$ is a complete semidistributive lattice \cite{DIRRT, GM}.

In our first main result we explicitly compute $\kappa$ for the lattice of torsion classes.
Hence, $\kappa$ is well-defined, even though $\tors \Lambda$ may be infinite.
We use the fact the map $M\mapsto \calT_M=\filt(\Gen(M))$ is a bijection from the set of isoclasses of bricks to the set of completely join-irreducible torsion classes $\CJI(\tors \Lambda)$ (see \cite[Theorem~1.0.3]{BCZ}).
Recall that $M$ is defined to be a \emph{brick} provided that $\End_\Lambda(M)$ is a division ring.
 \begingroup
\setcounter{tmp}{1} 
\renewcommand\theproposition{\Alph{tmp}}
 \begin{theorem}[Theorem~\ref{inpapermain1}]\label{main1}
Let $\Lambda$ be a finite dimensional algebra, let $M$ be a $\Lambda$-brick.
Then $\kappa:\CJI(\tors \Lambda)\to \CMI(\tors \Lambda)$ is a bijection with 
$$\kappa(\calT_M) = \leftidx{^\perp}{M}$$
where $\leftidx{^\perp}{M}$ denotes the set $\{X\in\module\Lambda | \Hom_{\Lambda}(X,M)=0\}$.
\end{theorem}
\endgroup
\setcounter{proposition}{\value{proposition}-1}

\begin{remark}
\normalfont
A very recent result of Reading, Speyer, and Thomas implies that $\kappa$ is a bijection for each \emph{completely} semidistributive lattice.
See \cite[Theorem~3.2]{RST}.
\end{remark}

Finite semidistributive lattices are analogous to the unique factorization domains of ring theory.
Each element $x\in L$ has a unique ``factorization'' in terms of the join operation which is irredundant and lowest, called the \emph{canonical join representation} and denoted by $x=\CJR(x)=\Join A$ (see Definition \ref{cjr}). 
The dual notion is the \emph{canonical meet representation}, denoted by $x=\CMR(x)=\Meet B$.
For finite semidistributive lattices, the map $\kappa$ has the pleasant property of respecting factorizations, i.e. $\kappa$ maps elements of any canonical join representation to the elements of some canonical meet representation, 
i.e. $\Meet \{\kappa(j) : j\in \CJR(x)\}=\CMR(y)$ for some $y\in L$.

With this fact in mind, we make the following definition of the $\bar\kappa$-map, which is an extension  of the $\kappa$-map to possibly infinite semidistributive lattices and also to a larger class of elements (not just completely join-irreducible elements).

\begin{definition}\label{extended_kappa}
Let $L$ be a (possibly infinite) semidistributive lattice. 
Let $x$ be an element which has a canonical join representation such that $\kappa(j)$ is defined for each $j\in\CJR(x)$. 
Define $$\bar{\kappa}(x) = \Meet \{\kappa(j): j\in \CJR(x)\}.$$ 
\end{definition}

It is clear that this generalized $\bar\kappa$ will not exist in general. 
In particular, if any $j\in \CJR(x)$ is not completely join-irreducible, then $\bar{\kappa}(x)$ is not defined. 
And even if all $j\in \CJR(x)$ are completely join-irreducible, $\kappa(j)$ might not exist for some $j$.
However when $\bar{\kappa}$ does exist it has quite nice properties, which we  
show 
in the next statements for the lattices of torsion classes.

 \begingroup
\setcounter{tmp}{2} 
\renewcommand\theproposition{\Alph{tmp}}
\begin{proposition}[Corollary~\ref{kappa tors}] 
Let $\Lambda$ be a finite dimensional algebra. Let 
$\calT$ be a torsion class which has a canonical join representation of the following form: $\CJR(\calT)=\Join_{\alpha\in A} \calT_{M_\alpha}$, where $M_{\alpha}$ are $\Lambda$-bricks. 
Then $\bar\kappa(\calT)$ is defined and is of the form:
\[\bar{\kappa}(\calT) = \bigcap_{\alpha\in A} \leftidx{^\perp}M_\alpha.\]
\end{proposition}
\endgroup

In order for $\bar\kappa$ to be defined as in this proposition, it is important that the torsion class $\calT$ has this nice canonical join representation. This will be important in the next theorem as well, so we define the following subset of torsion classes:
$$\tors_0(\Lambda):= \{\calT: \CJR(\calT)=\Join_{\alpha\in A} \calT_{M_\alpha}\text{where }  \{M_\alpha\}_{\alpha\in A} \text{ is a set of } \Lambda\text{-bricks}\}.$$

We also want to consider the compositions of $\bar\kappa$.
As one may expect, for general semidistributive lattice $L$, even if $\bar\kappa(x)$ exists, $\bar{\kappa}(\bar\kappa(x))$ 
might not be defined. And it may not be defined even when we consider lattice 
$\tors \Lambda$, since $\bar{\kappa}(\bar\kappa(\calT))$ might not be defined if $\bar{\kappa}(\calT)$ is not in $\tors_0(\Lambda)$.

A nice situation is when  $\Lambda$ is $\tau$-tilting finite. Then the lattice $\tors \Lambda$ is finite \cite{IRTT} and semidistributive.
In this case, $\bar\kappa^n$ is well-defined for all $n$ because every  element in $\tors \Lambda$ (and hence $\bar{\kappa}(\calT)$) has canonical join representation consisting of completely join-irreducibles (see Proposition \ref{CJR-join-irreducibles}). So by \cite{BCZ}  
$\CJR(\bar{\kappa}(\calT))=\Join_{\alpha\in A} \calT_{M_\alpha}$ and therefore is in $\tors_0 \Lambda$.
We consider this special case in Theorem E.

Now, we concentrate on hereditary finite dimensional algebras, where the structure of module categories is very well understood.
Let $H$ be a hereditary finite dimensional algebra. 
We show that $\bar\kappa$ sends functorially finite torsion classes to functorially finite torsion classes (Corollary \ref{k ftors}).
Since the class of functorially finite torsion classes is contained in $\tors_0\Lambda$ by
 Proposition \ref{first inclusion}, it follows that
 $\bar{\kappa}^n$ is well defined for functorially finite torsion classes for hereditary algebras.

We let $\tau$ denote Auslander-Reiten translation.
Our next result relates $\bar\kappa$-map to the Auslander-Reiten translation for hereditary algebras. 

 \begingroup
\setcounter{tmp}{3} 
\renewcommand\theproposition{\Alph{tmp}}
\begin{theorem}[Corollary~\ref{ar noninj}]\label{kappa_tau}
Let $H$ be a hereditary finite dimensional algebra and $M$ a brick such that $\calT_M$ is functorially finite. Then
\[\bar{\kappa}^2(\calT_M) = \calT_{\bar\tau^{-1} M}.\]
Here $\bar\tau^{-1} M=\tau^{-1} M$ for non-injective modules $M$ and $\bar\tau^{-1} I(S)=P(S)$ where $I(S)$ and $P(S)$ are the injective envelope and projective cover of the same simple $S$.
\end{theorem}
\endgroup 
\setcounter{proposition}{\value{proposition}-1}

Our final theorem shows that the $\bar\kappa$-map on torsion categories is essentially Ringel's $\epsilon$-map on the wide subcategories.
In~\cite{R}, Ringel defined a map $\epsilon:\wide H\to \wide H$ which generalizes a well-studied automorphism of the noncrossing partition lattice to the setting of wide subcategories.
$$\epsilon(\calW):=\leftidx{^{\perp_{0,1}}}\calW \text{ where}$$
 $$\leftidx{^{\perp_{0,1}}}\calW=\{X\in\module H : \Ext^1_{H}(X,M)=0= \Hom_{H}(X,M),  \forall M\in\calW\}.$$
 For the relation between wide subcategories and torsion classes,
 we use
$\alpha: \tors H\to \wide H$ which is defined as 
 $$\alpha(\calT):= \{X\in \calT: \forall (g:Y\to X)\in \calT,\, \ker(g)\in \calT\}.$$

 \begingroup
\setcounter{tmp}{4} 
\renewcommand\theproposition{\Alph{tmp}}
\begin{theorem}[Theorem~\ref{k and d}]\label{main}
Let $H$ be a finite dimensional hereditary algebra.
Let $\alpha$ and $\epsilon$ be defined as above.
Then the following diagram commutes:
\begin{center}
\begin{tikzpicture}[scale=0.55]
  \node (A) at (0,0) {$\tors_0 H$}; 
  \node(B) at (4,0) {$\tors H$};
  \node (C) at (0,-3) {$\wide H$}; 
  \node (D) at (4,-3) {$\wide H$.};
    \node at (2,-2.75) {$\epsilon$};
    \node at (2,.25) {$\bar{\kappa}$};
    \node at (-.25,-1.5) {$\alpha$};
    \node at (4.25,-1.5) {$\alpha$};
  \draw[->] (A.east)--(B.west);
  \draw[->] (A.south)--(C.north);
  \draw[->] (C.east)--(D.west);
  \draw[->] (B.south)--(D.north);
\end{tikzpicture}
\end{center}
\end{theorem}
\endgroup 
\setcounter{proposition}{\value{proposition}-1}
\begin{remark}
\normalfont
The dynamics of the $\bar\kappa$-map have been studied in \cite{TW1, TW2} in the special case when $L$ is finite, semidistributive and \emph{extremal}.
Such lattices are also called \emph{trim}.
When $\Lambda$ has finite-representation type and no cycles, 
 then $\tors \Lambda$ is extremal. 
In this setting, it was shown that $\bar\kappa$ can be factored as a composition of local moves called ``flips''.
See \cite[Theorem~1.1 and Corollary~1.5]{TW1}.
\end{remark}

\begin{remark}\label{Kreweras}
\normalfont
Let $W$ be a finite Weyl group.
When $H$ is hereditary and of Dynkin type $W$, then  the lattice of wide subcategories is isomorphic to the \emph{noncrossing partition lattice} $\NC(W)$ \cite{IS, IT, R}.
Under this isomorphism, the inverse of Ringel's map $\epsilon$ (called $\delta$ in \cite{R}) is equal to a classical lattice automorphism known as the \emph{Kreweras complement}, which we write here as $\Kr: \NC(W)\to \NC(W)$.
Theorem~\ref{main} implies that $\bar{\kappa}^{-1}$ is essentially equal to $\Kr$.
The combinatorics of the Kreweras complement are well-studied. 
For example, $Kr^{2h}$ is equal to the identity, where $h$ is the \emph{Coxeter number} of $W$.
For any orbit $\calO$ we have 
\begin{equation}\label{formula}
\frac{1}{|\calO|} \sum_{P\in \calO}|P| = r/2
\end{equation}
where $|P|$ is the number parts in $P$, and $r$ is the number of simples in $\Lambda$.
\end{remark}

When $\Lambda$ is $\tau$-tilting finite, we have an analogue of Formula~\ref{formula} in terms of $\bar{\kappa}$.
 \begingroup
\setcounter{tmp}{5} 
\renewcommand\theproposition{\Alph{tmp}}
\begin{theorem}\label{Kr_theorem}
Let $\Lambda$ be $\tau$-tilting finite, and let $r$ be the number of simples in $\module\Lambda$.
For any $\calT\in \tors \Lambda$ let $|\calT|$ denote the number of canonical joinands of $\calT$.
Then for any $\bar{\kappa}$-orbit $\calO$ we have
\[\frac{1}{|\calO|} \sum_{\calT\in \calO}|\calT| = r/2\]
\end{theorem}
%
\endgroup

The paper is organized in the following way: 
In Section~\ref{torsion sec} we review torsion classes, and a useful labeling of the cover relations in $\tors \Lambda$ by \emph{minimal extending modules}.
In Section~\ref{lattice sec} we review the necessary lattice-theoretic background.
In Section~\ref{kappa sec} we explicitly compute $\kappa$ and $\bar{\kappa}$ .
In Section~\ref{main sec}, we prove our main results relating $\kappa$ to wide subcategories. 
Finally, in Section~\ref{sec:AR} we relate $\bar{\kappa}$ to Auslander-Reiten translation.



\section{Torsion Classes Background}\label{torsion sec}

\subsection{Torsion pairs}

In this section we recall the definition of a torsion pair.
Throughout the section $\Lambda$ is a finite dimensional algebra. 
Denote by $\module \Lambda$ 
the category of finitely generated modules over $\Lambda$. 


\begin{definition}\label{torsion_pair}
A pair $(\calT, \calF)$ of full subcategories of $\module \Lambda$ is a \emph{torsion pair} if each of the following holds:
\begin{enumerate}
\item $\Hom_{\Lambda} (M,N) = 0$ for all $M\in \calT$ and $N\in \calF$.
\item $\Hom_{\Lambda}(M,-)|_\calF  = 0$ implies that $M\in \calT$.
\item $\Hom_{\Lambda} (-,N)|_\calT = 0$ implies that $N\in \calF$.
\end{enumerate}
\end{definition}
A full subcategory $\calT$ is a torsion class if and only if $(\calT, \calT^\perp)$ is a torsion pair \cite[Proposition~V.I.1.4]{ASS}.


For each torsion pair $(\calT,\calF)$ and each module $M\in \module \Lambda$ there is a canonical short exact sequence 
\[0\to tM\to M\to M/{tM}\to 0\]
such that $tM\in \calT$ and $M/{tM}\in \calF$ \cite[Definition~V.I.1.3]{ASS}.
We will use this fact in Section~\ref{wide} to make an important connection to the wide subcategories of $\Lambda$.

Throughout the paper, we study the lattice (poset) of torsion classes also denoted $\tors \Lambda$ in which $\calS \le \calT$ whenever $\calS \subseteq \calT$. 
More details are given in Section \ref{lattice sec}.

\subsection{Minimal inclusions of torsion classes}
In this section we consider minimal inclusions of torsion classes.
Each minimal inclusion is a cover relation or an edge in the Hasse diagram for the poset $\tors \Lambda$.
We also review a certain labeling of these edges by brick modules introduced independently in \cite{A, BCZ, DIRRT}.

\begin{definition}\label{def:min-ext-module}\cite[Definition 1.0.1]{BCZ}
A module $M$ is a \newword{minimal extending module for $\calT$} if it satisfies the following three properties:
\begin{enumerate}[label={(P\arabic*)}]
 \item Every proper factor of \label{propertyone}
   $M$ is in $\calT$;
  \item If $0\rightarrow M\rightarrow X
    \rightarrow T\rightarrow 0$ is a non-split exact sequence with
    $T\in \calT$, then $X\in \calT$; \label{propertytwo}
  \item $\Hom(\calT, M)=0$. \label{propertythree}
  \end{enumerate}
  Denote by $\ME(\calT)$ the set of isoclasses of minimal extending modules of~$\calT$.
\end{definition}
\begin{figure}[h]
\centering
\scalebox{.7}{
\begin{tikzpicture}
\draw[fill] (2,0) circle [radius=.1];
\draw[fill] (.5,1) circle [radius=.1];
\draw[fill] (3.5,1) circle [radius=.1];
\draw[fill] (.5,3) circle [radius=.1];
\draw[fill] (2,4) circle [radius=.1];

\draw [black, line width=.65 mm] (2,0) to (.5,1);
\draw [black, line width=.65 mm] (2,0) to (3.5,1);
\draw [black, line width=.65 mm] (.5,1) to (.5,3);
\draw [black, line width=.65 mm] (2,4) to (.5,3);
\draw [black, line width=.65 mm] (2,4) to (3.5,1);

\node [below] at (2,-.25) {$0$};
\node [left] at (.25,1) {$\add(S_1)$};
\node [left] at (.25,3) {$\add(S_1, P_1)$};
\node [above] at (2,4.25) {$\module \Lambda$};
\node [right] at (3.75,1) {$\add(S_2)$};

\node [left] at (1.25,.25) {$S_1$};
\node [left] at (1.25,3.65) {$S_2$};
\node [left] at (.55,2.1) {$P_1$};
\node [right] at (2.75,.25) {$S_2$};
\node [right] at (2.75,2.7) {$S_1$};

\end{tikzpicture}}
\caption{\label{fig_brick_labeling} We label edges 
of the lattice
$\tors A_2$ 
(of the hereditary algebra $kA_2$ where $A_2 =(1\rightarrow 2)$)
with the corresponding minimal extending modules.}
\end{figure}
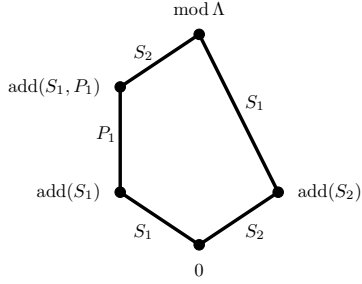 
Minimal extending modules arise in studying the minimal inclusions of torsion classes. 
We say that $\calS$ \emph{covers} $\calT$, denoted by $\calS\covers \calT$, if $\calT\subsetneq \calS$ and there is no torsion class $\calY$ such that $\calT\subsetneq\calY\subsetneq S$.
Each pair $\calS \covers \calT$ is a \emph{cover relation}.
(See Definition~\ref{def:cover}.)
For any cover relation $\calS\covers\calT$, there is a unique minimal extending module $M$ such that $\calS$ is the closure of $\calT\bigcup \{M\}$ taking iterative extensions.  

For any set $\calM$ of $\module \Lambda$, we write
$\filt\calM$ for the set of modules $M$ such that there is a sequence of submodules $0=M_0\subsetneq M_1\subsetneq \cdots \subsetneq M_l=M$ and the quotients $M_{i+1}/M_i\in\ind\calM$ for each $i\in \{1,2\ldots, i-1\}$.

\begin{theorem}\label{main: covers}\cite[Theorem 1.0.2]{BCZ}
Let $\calT$ be a torsion class over $\Lambda$ and let $[M]$ be the isoclass of the module $M$.
The map $$\eta_\calT: [M]\mapsto \filt(\calT\cup \{M\})$$ 
defines a set bijection $\eta_\calT:  \{\ME(\calT)\} \longleftrightarrow \{\calS \in \tors \Lambda\ |\ \calS\covers \calT\}$.
\end{theorem}
\begin{remark}\label{factors}
\normalfont 
The reason that $\filt(\calT \cup M)$ is a torsion class, when $M$ is a minimal extending module for $\calT$, is because each proper factor of $M$ belongs to $\calT$.
See \cite[Lemma~2.3]{BCZ}.
\end{remark}

A module $M$ is a minimal extending module for some torsion class if and only if it is a brick  \cite[Theorem 1.0.3]{BCZ}. 
We often visualize the minimal extending modules as labeling each of the edges in the poset $\tors \Lambda$ as in Figure~\ref{fig_brick_labeling}.

The notion of minimal extending module for torsion classes can be dualized to minimal co-extending module for torsion-free classes. 
\begin{definition}\label{def:min-coext-module}\cite[Definition 2.3.1]{BCZ}
A module $M$ is a minimal co-extending module for a torsion-free class $\calF$ if it satisfies the following three properties:
\begin{enumerate}[label={(P\arabic*')}]
 \item Every proper submodule of \label{propertyoneprime}
   $M$ is in $\calF$;
  \item If $0\rightarrow F\rightarrow X
    \rightarrow M\rightarrow 0$ is a non-split exact sequence with
    $F\in \calF$, then $X\in \calF$; \label{propertytwoprime}
  \item $\Hom(M, \calF)=0$. \label{propertythreeprime}
  \end{enumerate}
  \end{definition}

We omit the dual statement of Theorem \ref{main: covers} for minimal co-extending modules. 
It can be obtained easily from the following correspondence between minimal extending and co-extending modules.
Note that $\calS\covers \calT$ in $\tors\Lambda$ if and only if $\calT^\perp \covers \calS^\perp$ in $\torf\Lambda$.

\begin{proposition}\label{prop:torsion-free-to-torsion-relation}\cite[Proposition 2.3.3]{BCZ}
Suppose that $\calT\in \tors \Lambda$ and $M$ is
an indecomposable $\Lambda$ module such that $\filt(\calT \cup \{M\})$
is a torsion class. Then $M$ is a minimal extending module for the torsion class $\calT$ if and only if it is a minimal co-extending module for the torsion-free class $\filt(\calT \cup\{M\})^\perp$.
\end{proposition}


\begin{remark}\label{min_coext_rmk}
\normalfont
Property~\ref{propertythreeprime} implies that if $M$ is a minimal coextending module for the torsion-free class $\calT^\perp$, then $M\in \calT$.
\end{remark}

\section{Lattice Background}\label{lattice sec}
In this section we fill in the necessary background material on lattices.
We recall basic definitions, examples, computations and facts about: complete semidistributive lattices, canonical join and meet representations, completely join-irreducible and meet-irreducible elements.

\subsection{Complete Semidistributive Lattices} 

Recall that a lattice (or lattice-poset) is a poset in which each pair of elements has a unique smallest upper bound $\Join$ (called the join) and a unique greatest lower bound $\Meet$ (called the meet). A lattice $L$ is \emph{complete} if the elements $\Join A$ and $\Meet A$ exist for any (possibly infinite) set $A\subseteq L$. In particular, $\Join L$ and $\Meet L$ exist.
Therefore a complete lattice $L$ has a unique greatest element $\1$ and a unique smallest element $\0$. An example of a lattice which is not complete is the Divisibility Poset (see Example \ref{Divisibility}).

\begin{definition}\label{def}
A lattice (lattice-poset) $L$ is called \emph{join-semidistributive} if it satisfies~\ref{jsd} for every triple $x,y,z$ in $L$.
Dually, a lattice is \emph{meet-semidistributive} if it satisfies~\ref{msd} for every $x,y,z$ in $L$.
We say that $L$ is \emph{semidistributive} if for each $x,y,z\in L$, both ~\ref{jsd}  and ~\ref{msd}  are true.
\begin{equation}\label{jsd}
\text{If $x\join y = x\join z$, then $x\join(y\meet z) = x\join y$}\tag{$SD_\join$}
\end{equation}
\begin{equation}\label{msd}
\text{If $x\meet y = x\meet z$, then $x\meet(y\join z) = x\meet y$}\tag{$SD_\meet$}
\end{equation}
\end{definition} 

\begin{example}\label{Divisibility}
\normalfont
Let $L$ be the Divisibility Poset, whose elements are the positive integers $\mathbb N$ ordered $x<y$ whenever $x|y$.
Observe that $x\join y = \lcm(x,y)$ and $x\meet y = \gcd(x,y)$.
However, $L$ is \emph{not} complete. Indeed there is no largest element~$\1$.
Observe that the join and meet operations distribute: $x\join( y\meet z) = (x\join y)\meet (x\join z)$ and $x\join ( y\meet z) = (x\meet y)\join (x\meet z)$ for all $x,y,z\in \mathbb{N}$.
Therefore  $L$ is semidistributive but not complete. 
\end{example}

\begin{figure}[h]
\centering
\scalebox{.7}{
\begin{tikzpicture}
\draw[fill] (2,0) circle [radius=.1];
\draw[fill] (.5,1) circle [radius=.1];
\draw[fill] (3.5,1) circle [radius=.1];
\draw[fill] (.5,3) circle [radius=.1];
\draw[fill] (2,4) circle [radius=.1];

\draw [black, line width=.65 mm] (2,0) to (.5,1);
\draw [black, line width=.65 mm] (2,0) to (3.5,1);
\draw [black, line width=.65 mm] (.5,1) to (.5,3);
\draw [black, line width=.65 mm] (2,4) to (.5,3);
\draw [black, line width=.65 mm] (2,4) to (3.5,1);

\node [below] at (2,-.25) {$0$};
\node [left] at (.25,1) {$\add(S_1)$};
\node [left] at (.25,3) {$\add(S_1, P_1)$};
\node [above] at (2,4.25) {$\module \Lambda$};
\node [right] at (3.75,1) {$\add(S_2)$};

\end{tikzpicture}
\qquad\qquad\qquad
\begin{tikzpicture}
\draw[fill] (2,0) circle [radius=.1];
\draw[fill] (-1,2) circle [radius=.1];
\draw[fill] (5,2) circle [radius=.1];
\draw[fill] (2,2) circle [radius=.1];
\draw[fill] (2,4) circle [radius=.1];

\draw [black, line width=.65 mm] (2,0) to (-1,2);
\draw [black, line width=.65 mm] (2,0) to (5,2);
\draw [black, line width=.65 mm] (2,0) to (2,2);
\draw [black, line width=.65 mm] (2,4) to (2,2);
\draw [black, line width=.65 mm] (2,4) to (5,2);
\draw [black, line width=.65 mm] (2,4) to (-1,2);

\node [below] at (2,-.25) {$0$};
\node [left] at (-1.25,2) {$\add(S_1)$};
\node [left] at (1.75,2) {$\add(P_1)$};
\node [above] at (2,4.25) {$\module \Lambda$};
\node [right] at (5.25,2) {$\add(S_2)$};
\end{tikzpicture}
}
\caption{The lattice of torsion classes (left), and the lattice of wide subcategories (right) for a type $A_2$ hereditary algebra $1\rightarrow 2$.}
\label{torsion_vs_wide}
\end{figure}
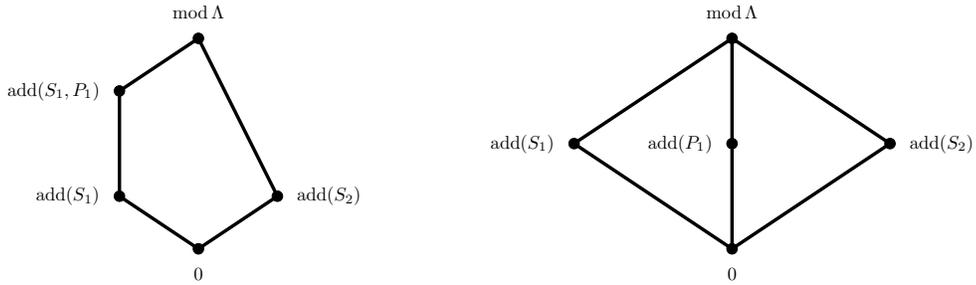

\begin{example}\label{finite_tors}
\normalfont
We check~\ref{jsd} \!in a small example for the lattice of torsion classes shown left in Figure~\ref{torsion_vs_wide}.
Take $x=\add(S_2)$, $y=\add(S_1)$ and $z=\add(S_1, P_1)$.
Observe that $x\join y =\add(S_2)\join\add(S_1)=\add(S_2)\join\add(S_1, P_1)= x \join z$.
Also, $y\meet z =\add(S_1)\meet \add(S_1, P_1) = \add(S_1) \cap \add(S_1, P_1) = \add(S_1).$
Therefore,
$x\join [y\meet z] = \add(S_2) \join [\add(S_1) \meet \add(S_1, P_1)] = \add(S_2) \join \add(S_1) = x\join y.$
One can check 
that this lattice is indeed semidistributive.

On the other hand, the lattice of wide subcategories shown to the right is not semidistributive.
Note that $\add(S_2)\join \add(S_1) = \add(S_2) \join \add(P_1)$.
However $\add(P_1)\meet \add(S_1) = 0$, so the lattice fails~\ref{jsd}.

\end{example}

\begin{example}
\normalfont
We display an example of the lattice of torsion classes for the path algebra of the Kronecker quiver with vertices $Q_0=\{1,2\}$ and arrows 
$Q_1=\{\alpha,\beta\}$:
$$\begin{tikzpicture}
\node (A) at (0,0) {$Q=(2\  \cdot$};
\node (B) at (2,0) {$ \cdot \ 1)$};
\draw[->] ([yshift=-.2em] A.east)--([yshift=-.2em]B.west) node[midway,below] {$\beta$};
  \draw[->] ([yshift=.2em] A.east)--([yshift=.2em] B.west) node[midway,above]
  {$\alpha$};
  \end{tikzpicture} $$
 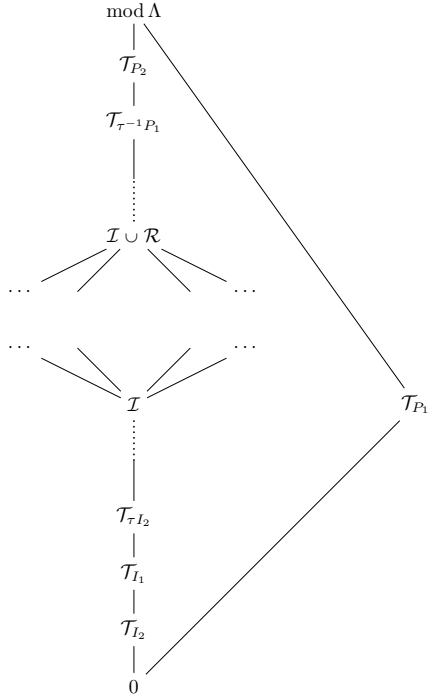
\begin{figure}[h]
  \scalebox{.75}{
  \begin{tikzpicture}
 \node(0) at (0,-5) {0};
 \node(1) at (0,-4) {$\calT_{I_2}$};
  \node(2) at (0,-3) {$\calT_{I_1}$};
   \node(3) at (0,-2) {$\calT_{\tau I_2}$};
   
   \node(r0) at (0,0) {$\calI$};
    
  \draw[-] (0)--(1)--(2)--(3)--(0,-1); 
  
  \draw[thick,dotted] (0,-1)--(r0);
  
  \node(r11) at(-2,1) {$\cdots$};
  \node (r12) at (2,1) {$\cdots$};
  
  \draw[-] (r0)--(-1,1);
    \draw[-] (r0)--(1,1);
      \draw[-] (r0)--(r11);
        \draw[-] (r0)--(r12);
     
     \node(r21) at(-2,2) {$\cdots$};
  \node (r22) at (2,2) {$\cdots$};

   \node (r1) at (0,3) {$\calI\cup\calR$};
     \draw[-] (r1)--(-1,2);
    \draw[-] (r1)--(1,2);
      \draw[-] (r1)--(r21);
        \draw[-] (r1)--(r22);
     
\draw[thick, dotted] (r1)--(0,4);

\node(4) at (0,5) {$\calT_{\tau^{-1}P_1}$};
\node(5) at (0,6) {$\calT_{P_2}$};
\node(6) at (0,7) {$\module \Lambda $};
    
 \draw(0,4)--(4)--(5)--(6);
 
 \node(p1) at (5,0){$\calT_{P_1}$};
 \draw(0)--(p1)--(6);

   \end{tikzpicture}}
   \caption{The lattice of torsion classes for the path algebra of Kronecker quiver. Here $\calI$ stands for preinjective modules and $\calR$ for regular modules.}
  \label{fig:kronecker}
  \end{figure}


The lattice of torsion classes for the Kronecker quiver from Figure~\ref{fig:kronecker} is an example of an infinite complete semidistributive lattice.
It is well-known that for any finite dimensional algebra $\Lambda$, $\tors \Lambda$ is a complete semidistributive lattice \cite{DIRRT, GM} with meet $\calT\meet \calS= \calT\cap \calS$ and join $\calT\join \calS= \filt(\calT\cup\calS)$.
\end{example}

\begin{remark}\label{antiso}
\normalfont
Suppose that $\calT$ and $\calS$ are torsion classes.
Observe that $\calT\subseteq \calS$ if and only if $\calS^\perp \subseteq \calT^\perp$.
We write $\torf \Lambda$ for the lattice of torsion-free classes ordered by containment.
It follows immediately that the lattices $\tors \Lambda$ and $\torf \Lambda$ are anti-isomorphic.
\end{remark}


\subsection{Canonical Join and Meet Representations: CJR and CMR}
Here we review the definition of join representations, canonical join representations, and we make the connection to semidistributivity.

A join representation of an element $x\in L$ is an expression $\Join A =x$.
We say that the join representation is \emph{irredundant} provided that $\Join A' < x$ for each proper subset of $A$.
We partially order the set of irredundant join representations of $x$ as follows:
Say that $\Join A$ is ``lower'' than $\Join B$ provided that for each $a\in A$ there is some element $b\in B$ such that $a\le b$.

\begin{definition}\label{cjr}
The \emph{canonical join representation of $x$}, denoted by $\CJR(x)$ is the unique lowest irredundant join representation among all irredundant join representations of $x$, when such a join representation exists.

By convention we set $\CJR(\0)=\Join \emptyset$ to be the empty join.
The canonical meet representation is defined dually, and denoted $\CMR(x)$.
\end{definition}

\begin{theorem}\cite[Theorem~2.24]{FJN} \label{thm_ext_cjr}
Suppose that $L$ is a finite lattice.
Then $L$ satisfies \ref{jsd} (join semi-distributive law) if and only if each element in $L$ has a canonical join representation.
Dually, $L$ satisfies \ref{msd} (meet semi-distributive law) if and only if each element in $L$ has a canonical meet representation.
\end{theorem}


\begin{example}\label{number_theory}
\normalfont
Let $L$ be the divisibility poset as in Example \ref{Divisibility}. 
The canonical join representation coincides with primary decomposition:
 \[x=\CJR(x)=\Join \{p^d: \text{$p$ is prime and $p^d$ is the largest power of $p$ dividing $x$}\}.\] 
\end{example}

\begin{example}\label{CJR}
\normalfont
Let us compute the canonical join representation of $\module \Lambda$ in the lattice of torsion classes from Example~\ref{finite_tors} and shown in Figure~\ref{fig_brick_labeling}.
Observe that both $\add(S_1,P_1) \join \add(S_2) = \module \Lambda = \add(S_1)\join \add(S_2)$ are join representations for $\module \Lambda$.
However the torsion class $\add(S_1)$ is ``lower'' than $\add(S_1, P_1)$.
Therefore $\add(S_1)\join \add(S_2)$ is the canonical join representation for $\module \Lambda$.

In the lattice of wide subcategories, as in Figure \ref{torsion_vs_wide}, the canonical join representation of $\module \Lambda$ does not exist because there is not a \emph{unique} lowest join representation. 
\end{example}

The next result connects minimal extending modules to canonical join representations in $\tors \Lambda$.
\begin{proposition}\label{cjr_and_brick_labeling}\cite[Theorem~1.08 and Lemma~3.2.4]{BCZ}
Let $\Lambda$ be a finite dimensional algebra and $\calT\in \tors \Lambda$.
Assume that $\calT$ has canonical join representation $$\CJR(\calT)=\Join_{\alpha\in A} \filt(\Gen(M_{\alpha})).$$  
Then for each $\alpha\in A$, there is a torsion class $\calY_\alpha \covered \calT$ such that $M_\alpha$ is the minimal extending module for this cover relation.
\end{proposition}

\begin{remark}
\normalfont
When $L$ is an infinite semidistributive lattice, the canonical join representation or the canonical meet representation may not exist, even when $L$ is additionally a complete lattice.
For example, consider the lattice depicted in Figure~\ref{non_ex}.
Each dashed line in the figure represents an infinite chain of elements, none of which cover $\0$ nor are covered by $\1$.
Observe that $L$ is both complete and semidistributive, however the canonical join representation of the top element of $L$ does not exist and the canonical meet representation for the bottom element does not exist.
\end{remark}

\begin{figure}[h]
 \centering
 \scalebox{.6}{
\begin{tikzpicture}

\draw[thick, dashed] (-0.25,0) to [out=135,in=225] (-0.25,4);
\draw[thick, dashed] (0.25,0) to [out=45,in=-45] (0.25,4);
\node at (0,4.25) {$\1$};
\node at (0,-.25) {$\0$};

\end{tikzpicture}}
 \caption{An infinite complete semidistributive lattice.
 Each dashed line represents an infinite chain in~$L$.}
        \label{non_ex}
\end{figure}
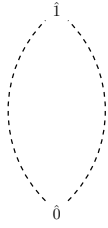

\subsection{Completely Join and Meet-Irreducible Elements CJI and  CMI}

\begin{definition}\label{join-irreducible} \normalfont{
An element $j\in L$ is join-irreducible if whenever $j=\Join A$ and $A$ is {\bf finite} then $j\in A$.
We say that $j\in L$ is \emph{completely join-irreducible} if whenever $j=\Join A$ then $j\in A$.
Completely meet-irreducible elements are defined dually.
By convention, we set the empty join to be $\0$, and hence $\0$ is not join-irreducible or completely join-irreducible.
Similarly, $\1$ is not meet-irreducible or completely meet-irreducible.}
\end{definition}

\begin{definition}\label{def:cover}
\normalfont
An element $y$ \emph{covers} $x$ if $y>x$ and there is no $z$ such that $y>z>x$.
In this case we also say that $x$ is \emph{covered by} $y$, and we use the notation $y\covers x$.
The pair $(x,y)$ is called a \emph{cover relation}.
\end{definition}

\begin{remark}\label{j_*}
\normalfont
Notice that an element $j$ is completely join-irreducible if it covers precisely one element which we write as $j_*$. 
\end{remark}

We write $\CJI(L)$ (respectively $\CMI(L)$) for the set of completely join-irreducible elements of $L$ (completely meet-irreducible elements of $L$).


\begin{example}\label{number_theory}
\normalfont
Let $L$ be the divisibility poset as in Example \ref{Divisibility}. 
An element $x$ is completely join-irreducible if and only if $x=p^n$ where $p$ is prime and $n$ is a positive integer.
Note that if $p$ is prime then $px\covers x$.
Therefore, none of the elements of $L$ are completely meet-irreducible (because each positive integer is covered by infinitely many elements).
\end{example}

\begin{example}\label{A2_ex}
\normalfont
Consider the lattice of torsion classes for hereditary algebra of type $A_2$ shown in Figure~\ref{fig_brick_labeling}.
Observe that $\add(S_1), \add(S_2)$ and $\add(S_1, P_1)$ are all join-irreducible and completely join-irreducible.
\end{example}

\begin{example}\label{ji_kronecker}
\normalfont
Consider the lattice of torsion classes for the Kronecker quiver shown in Figure~\ref{fig:kronecker}.
In the figure, $\calI$ denotes the subcategory of all preinjective modules.
Observe that $\calI$ is join-irreducible but it is not completely join-irreducible.
\end{example}

\begin{proposition} \label{CJR-join-irreducibles}
\normalfont (a) Let $\CJR(x)=\Join A$ be the canonical join representation of $x$. Then every $a\in A$ is join-irreducible. \\
(b) Let $L$ be a finite lattice. Let $\CJR(x)=\Join A$ be the canonical join representation of $x$. Then every $a\in A$ is completely join-irreducible.
\end{proposition}
\begin{proof} 
Suppose there is some ${a\in A}$ that is not join-irreducible.
Then $a=\Join A'$, where $A'=\{a'_1, a'_2, \ldots, a'_n\}$ is finite.
Thus $x= \Join (A\setminus \{a\}) \cup A'$ is a lower join representation.
We now create an irredundant lower join representation of $x$.

First, observe that for any $b\in A\setminus \{a\}$ we have $\Join (A\setminus \{a,b\}) \cup A'<x$ as shown below:
\begin{align*}
\Join (A\setminus \{a,b\}) \cup A' &= \Join (A\setminus \{a,b\}) \Join A'\\
&= \Join (A\setminus \{a,b\}) \join a\\
&=\Join A\setminus \{b\} <x 
\end{align*}
The last equality and inequality follow from the fact that $\Join A$ is irredundant.
Therefore, the only way $\Join (A\setminus \{a\}) \cup A' $ is not irredundant is if there is some other element $a'\in A'$ such that $\Join (A\setminus \{a\}) \cup (A'\setminus \{a'\}) = x$.
Remove each such element from $A'$ until $x=\Join A \setminus \{a\})\cup (A'\setminus \{a'_{i_1}, a'_{i_2}, \ldots, a'_{i_k}\})$ is irredundant.
This is possible because $A'$ is finite.
Thus we obtain an irredundant join representation of $x$ that is lower than the one we started with.
\end{proof}

\begin{remark}\ 
\normalfont
\begin{enumerate}
\item However in the canonical join representation $\CJR(x)=\Join A$,  $a\in A$ may not be completely join-irreducible.
For example, let lattice $L$ be defined by $\hat 0=a_0\covered a_1\covered \dots  \hat 1$, where $\1$ is greater than $a_i$ for all $i$, but $\1$ does not cover any elements. Then $\CJR(\hat 1)=\hat 1$.
However $\hat 1$ is not completely join irreducible since $\hat 1= \Join \{a_i | i\geq 0\}$ and $\hat 1\neq a_i$ for all $i$.
\item Notice, in general, if $j$ is completely join-irreducible then $\CJR(j) = \{j\}$.
\end{enumerate}
\end{remark}
We use the following theorem throughout the paper.

\begin{theorem}\label{schur modules to join-irreducibles} \cite[Theorem 3.1.1]{BCZ}
Let $\Lambda$ be a finite dimensional algebra. Then:
\begin{enumerate}
\item The map $\zeta:[M]\mapsto \calT_M=FiltGen(M)$ defines a bijection from the
set of isoclasses of bricks over $\Lambda$ to the set of completely join-irreducible torsion classes $\CJI(\tors \Lambda)$.
\item  The map $\zeta':[M]\mapsto FiltCogen(M)$ defines a bijection from the
set of isoclasses of bricks over $\Lambda$ to the set of completely meet-irreducible torsion-free classes $\CMI(\torf \Lambda)$.
\end{enumerate}
\end{theorem}

\begin{remark}\label{useful_rmk}
\normalfont
Let $M$ be a brick and let $(\calT_M)_*$ be the unique torsion class covered by $\calT_M$. It follows from  \cite[Proposition~2.3.4]{BCZ} that:
\begin{enumerate}
\item $M$ is a minimal extending module for the cover relation $(\calT_M)_*\covered \calT_M$. 
\item Each proper factor of $M$ is in $(\calT_M)_*$ by Property~\ref{propertyone}, and $M\in ((\calT_M)_*)^\perp$ by Property~\ref{propertythree}.
\end{enumerate}
\end{remark}

\section{Kappa Map}\label{kappa sec}

\subsection{Definition of Kappa Map}
Before we prove our first main result, we recall the fundamental object of the paper, the $\kappa$-map.
\begin{definition}\label{kappa-map with*}
Let $j$ be a completely join-irreducible element of a lattice $L$, and let $j_*$ be the unique element covered by $j$.
Define $\kappa(j)$ to be:
$$\kappa(j): = \text{unique max}\{x\in L: j_*\le x\text{ and } j\not\le x\},
\text{ when such an element exists.}$$
\noindent Dually, let $m$ be a completely meet-irreducible element and let $m_*$ be the unique element covering $m$. Define $\kappa_*(m)$ to be:
$$\kappa_*(m):=\text{unique min}\{x\in L: x\le m_*\text{ and } x\not \le m\}, \text{when such an element exists.}$$
\end{definition}


\begin{example}\label{ex1}
\normalfont
We evaluate the $\kappa$ map for the lattice of torsion classes displaced left in Figure~\ref{torsion_vs_wide}.
First consider $\kappa(\add(S_2))$.
Note that $\add(S_2)$ is completely join-irreducible because it covers precisely one element $0$. 
Since $\add(S_1, P_1)$ is largest torsion class such that $0\le \add(S_1, P_1)$ and $\add(S_2)\not \le \add(S_1, P_1)$, it follows that $\kappa(\add(S_2))= \add(S_1, P_1)$.
Similarly $\kappa(\add(S_1))= \add(S_2)$ and $\kappa(\add(S_1,P_1)) = \add(S_1)$.

On the other hand, the map $\kappa$ is not well-defined on the lattice of wide subcategories for $\Lambda$.
If we try to compute $\kappa(\add(S_1))$ we see that there are two incomparable maximal elements which contain the wide subcategory $0$, and do not contain $\add(S_1)$.
When $L$ is finite, then it is semidistributive if and only if $\kappa$ is a bijection from $\CJI(L)$ to $\CMI(L)$ with inverse $\kappa_*$ \cite[Corollary~2.55]{FJN}.
Recall, the lattice of wide subcategories of $\Lambda$ is not semidistributive.
\end{example}


\subsection{The Kappa Map and Minimal Extending Modules}
In this section, we consider arbitrary finite dimensional algebras and prove Theorem~\ref{main1},  i.e. the combinatorially defined $\kappa$-map (Definition~\ref{kappa-map with*}), when restricted to completely join-irreducible torsion classes, has a very nice representation theoretic interpretation as the left Hom-orthogonal operation.

We begin with a useful lemma relating minimal extending modules and naturally associated torsion classes obtained by taking Hom-orthogonal subcategories.
Recall that $\ME(\calT)$ is the set of isoclasses of minimal extending modules of $\calT$.
 \begin{lemma}\label{left perp}
Let $\Lambda$ be a finite dimensional algebra and $M$ be a brick. Then
\begin{enumerate}
\item $\leftidx{^\perp}{M}$ is a torsion class.
\item  $(\leftidx{^\perp}M)^\perp=\filt\Cogen M$.
\item $\ME(\leftidx{^\perp}M)=\{M\}$.
\end{enumerate}
\end{lemma}
\begin{proof}
(1) It is clear that $\leftidx{^\perp}M$ is closed under quotients and extensions. 
Therefore it is a torsion class.\\
(2) Since $M\in(\leftidx{^\perp}M)^\perp$ and $(\leftidx{^\perp}M)^\perp$ is a torsion free class, it follows that $\filt\Cogen M\subseteq (\leftidx{^\perp}M)^\perp$. On the other hand, since $M\in\filt\Cogen M$, it follows that $\leftidx{^\perp}M\supseteq \leftidx{^\perp}\filt\Cogen M$ and hence $(\leftidx{^\perp}M)^\perp\subseteq (\leftidx{^\perp}\filt\Cogen M)^\perp=\filt\Cogen M$, where the last equality holds because $\filt\Cogen M$ is a torsion free class. \\
(3) By the Theorem \ref{schur modules to join-irreducibles}(2), $\filt\Cogen(M)$ is a completely join-irreducible torsion-free class. 
Hence, by Proposition \ref{prop:torsion-free-to-torsion-relation}, $M$ is the unique minimal extending module for $\leftidx{^\perp}M$. 
\end{proof}

\subsection{Kappa map and Proof of Theorem A}
We are now prepared to prove Theorem~\ref{main1}, which we restate below.

\begin{theorem}[Theorem~\ref{main1}]\label{inpapermain1}
Let $\Lambda$ be a finite dimensional algebra, $M$  a $\Lambda$-brick and $\calT_M=\filt(\Gen(M))$ the completely join-irreducible torsion class.
Then $$\kappa(\calT_M) = \leftidx{^\perp}{M}.$$
Moreover, $\kappa:\CJI(\tors \Lambda)\to \CMI(\tors \Lambda)$ is a bijection between the completely join-irreducible and completely meet irreducible torsion classes.
\end{theorem}

\begin{proof} 
Let $M$ be a brick.
Part~3 of Lemma~\ref{left perp} implies that $\leftidx{^\perp}M$ is completely meet irreducible.
Observe that $M\notin \leftidx{^\perp}M$ so $\calT_M\not\le \leftidx{^\perp}M$.
Let $(\calT_M)_*$ denote the unique element covered by $\calT_M$.
We claim that $(\calT_M)_*\le \leftidx{^\perp}M$.
Recall from Remark~\ref{useful_rmk}, $M\in ((\calT_M)_*)^\perp$.
Since $\filt\Cogen M$ is the smallest torsion-free class containing $M$, we have $ \filt\Cogen M = (\leftidx{^\perp}{M})^\perp\le ((\calT_M)_*)^\perp$.
Applying the lattice anti-isomorphism $\leftidx{^\perp}(-):\torf\Lambda\to\tors\Lambda$, we get $ (\calT_M)_*\le \leftidx{^\perp}{M}$.
This proves our claim.

Let $\calS$ be any torsion class satisfying
\begin{itemize}
\item $(\calT_M)_*\le  \calS$ and
\item $\calT_M \not \le \calS$.
\end{itemize}
We claim that $\calS\le \leftidx{^\perp}{M}$.
The first item implies that every proper factor of $M$ is in $\calS$ (because every proper factor of $M$ is in $(\calT_M)_*$; see Remark~\ref{useful_rmk}).
The second item implies that $M\not\in \calS$.
Since $\calS$ is closed under extensions, we conclude that $M\in \calS^\perp$.
Therefore $\filt\Cogen M \le \calS^\perp$.
As in the previous paragraph, we conclude that $\calS\le \leftidx{^\perp}{M}$.
Therefore $\kappa(\calT_M) = \leftidx{^\perp}M$.

It is clear that $\kappa$ is one-to-one.
Part~2 of Lemma~\ref{left perp} implies that each completely meet-irreducible torsion class is $\leftidx{^\perp}M$ for some brick $M$.
Therefore $\kappa$ is also surjective.
\end{proof}

\subsection{Extension of Kappa Map}
We now recall the definition of $\bar\kappa$ map from \ref{extended_kappa}
\begin{definition}
Let $L$ be a (possibly infinite) semidistributive lattice. 
Let $x$ be an element which has a canonical join representation such that $\kappa(j)$ is defined for each $j\in\CJR(x)$ and that $\Meet \{\kappa(j): j\in \CJR(x)\}$ exists.
Define $$\bar{\kappa}(x) = \Meet \{\kappa(j): j\in \CJR(x)\}.$$ 
\end{definition}

\begin{example}
\normalfont
Let $\Lambda$ be the hereditary algebra of type $A_2$ from Examples~\ref{ex1}, and observe that 
\begin{align*}
\bar{\kappa} (\module \Lambda) =& \Meet \{\kappa(\add(S_1)), \kappa(\add(S_2))\}\\
=& \Meet \{\add(S_2), \add(S_1,P_1)\}\\
=&\add(P_1,S_1) \cap \add(S_2)\\
=&0.
\end{align*}
\end{example}

\begin{corollary}\label{kappa tors} 
Let $\Lambda$ be a finite dimensional algebra, and let $\calT$ be a torsion class with canonical join representation $\Join_{\alpha\in A} \calT_{M_\alpha}$, where $\{M_\alpha\}$ is a set of 
bricks. 
Then:
$$\bar\kappa(\calT)= \bigcap_{\alpha\in A} \leftidx{^\perp}{M_\alpha}.$$
\end{corollary}

Note that the minimal extending modules of $\bar{\kappa}(\calT)$ are precisely bricks $M_\alpha$ appearing in the canonical join representation of $\calT$.
We restate this result in terms of $\Hom$-orthogonal bricks using the following theorem.
\begin{theorem}\cite[Theorem~1.0.8]{BCZ}\label{torsion_cjr}
Let $\Lambda$ be a finite dimensional algebra, and suppose that $\calE$ is a collection of bricks over $\Lambda$.
Then $\Join \{\calT_M:\, M\in \calE\}$ is a canonical join representation of some torsion class if and only if 
$\Hom_{\Lambda}(M, M') = \Hom_{\Lambda}(M', M) =0$ for each pair $M, M'\in \calE$.  \end{theorem}
\begin{proposition}\label{left perp 2}
Let $\Lambda$ be an artin algebra and $\{M_\alpha\}_{\alpha\in A}$ be $\Hom$-orthogonal bricks.  \begin{enumerate}
\item $\bigcap_{\alpha\in A}\leftidx{^\perp}{M_\alpha}$ is a torsion class.
\item  $(\bigcap_{\alpha\in A}\leftidx{^\perp}M_\alpha)^\perp=\Join\filt\Cogen {M_i}$.
\item $\ME(\bigcap_{\alpha\in A}\leftidx{^\perp}M_\alpha)=\{M_\alpha\}_{\alpha\in A}$.
\end{enumerate} 
\end{proposition}
\begin{proof}
(1) follows from the fact that intersection of torsion classes is a torsion class. (2) Since $(-)^\perp:\tors\Lambda\to\torf\Lambda$ is a lattice anti-isomorphism, $(\bigcap\leftidx{^\perp}M_\alpha)^\perp=\Join(\leftidx{^\perp}M_\alpha)^\perp=\Join\filt\Cogen {M_\alpha}$. (3) follows from (2).
\end{proof}

\begin{remark}
\normalfont
Proposition~\ref{left perp 2} implies that $\bar{\kappa}$ can be computed directly from the labeling of the cover relations in $\tors \Lambda$ by minimal extending modules.
If $\{M_{\alpha}\}_{\alpha\in A}$ label the lower cover relations for a torsion class $\calT$, then $\{M_{\alpha}\}_{\alpha\in A}$ is precisely the set of minimal extending modules labeling the upper cover relations of $\bar{\kappa}(\calT)$.

\end{remark}
\begin{remark}\normalfont
For an infinite semidistributive lattice $L$, $\bar\kappa$ is not necessarily defined for all elements $x\in L$. The existence of $\bar\kappa(x)$ depends on: (1) the existence of $\CJR(x)$; 
(2) every element $j\in\CJR(x)$ being completely join-irreducible;
(3) the existence of $\kappa(j)$ for each element $j\in \CJR(x)$;
 (4) the existence of the meet $\Meet \{\kappa(j): j\in \CJR(x)\}$. 
\end{remark}
We will use the following two subclasses of $\tors\Lambda$, for the purpose of defining $\bar\kappa$ and also for the purpose of relating torsion classes to wide subcategories of $\module\Lambda$.
\begin{definition}
Let $\Lambda$ be a finite dimensional algebra. \\
Let $\ftors(\Lambda)$ be the set of functorially finite torsion classes over $\Lambda$.\\
Let $\tors_0(\Lambda)$ denote the set of torsion classes $\calT$ such such that  $\CJR(\calT)=\Join_{\alpha\in A} \calT_{M_\alpha}$ where $\{M_\alpha\}_{\alpha\in A}$ is a set of bricks.
\end{definition}

Since $\kappa$ is defined for completely join-irreducible elements---and these are exactly the torsion classes $\calT_M$ where $M$ is a brick---and since torsion classes are closed under arbitrary intersections, the set $\tors_0(\Lambda)$ consists of exactly all the torsion classes $\calT$ such that $\bar{\kappa}(\calT)$ exists.


We have the following inclusions of sets: 
\begin{proposition}\label{first inclusion}Let $\Lambda$ be a finite dimensional algebra. There are inclusions: 
$$\ftors(\Lambda)\subseteq \tors_0(\Lambda)\subseteq \{\calT\in\tors(\Lambda)| \CJR(\calT) \text{\ exists}\}\subseteq\tors(\Lambda).$$
\end{proposition}
 \begin{proof}
 We only need to prove the first inclusion. Let $\calT\in\ftors(\Lambda)$, i.e. let $\calT$ be a functorially finite torsion class. For the entire proof we will use the fact that there is a bijection between functorially finite torsion classes 
$\ftors(\Lambda)$ and support $\tau$-tilting modules (see \cite[Theorem 2.7]{AIR}).  Next (see \cite[Theorem 2.33]{AIR}) implies that all covering relations $\calT\covers \calT_i$ correspond to mutations of support $\tau$-tilting modules. The fact that for each $\tau$-tilting module, there are only finitely many mutations, implies that there are only finitely many 
functorially finite
$\calT_i$ such that $\calT\covers \calT_i$.
Now Theorem~\ref{main: covers} says that for each cover relation $\calT_i \covered \calT$, there is a unique minimal extending module $M_i$ such that $\calT = \filt(\calT_i\cup M_i)$.
(Alternatively, $M_i$ is a minimal co-extending module for the torsion-free class $\calT^\perp$.)
It follows from \cite[Theorem 1.3]{DIJ} that if $\calS< \calT$ in $\tors \Lambda$, there exists some $i$ such that $\calS\le\calT_i$.
This is exactly the hypothesis of \cite[Corollary~3.2.6]{BCZ}, 
which says that $\Join_{i\in I} \calT_{M_i}$ is the canonical join representation of $\calT$.
Therefore $\calT\in\tors_0(\Lambda)$
 \end{proof}
 
\begin{remark}\normalfont
There exist $\Lambda$ such that the first two inclusions are proper (see Example \ref{Kronecker torsions}). It is still an open question if the last containment above can be made proper.  
\end{remark}

\begin{example}\label{Kronecker torsions}
\normalfont
Let $\Lambda= kQ$ where $Q$ is the Kronecker quiver from Figure~\ref{fig:kronecker}.
Denote by $\calR$ and $\calI$ the subcategory of regular modules and preinjective modules respectively. 
 Denote by $\{R_\lambda\}$ the regular module  \[\begin{tikzpicture}
\node (A) at (0,0) {k};
\node (B) at (2,0) {k};
\draw[->] ([yshift=-.2em] A.east)--([yshift=-.2em]B.west) node[midway,below] {$[m]$};
  \draw[->] ([yshift=.2em] A.east)--([yshift=.2em] B.west) node[midway,above]
  {$[n]$};
  \end{tikzpicture}\] where $(m:n)=\lambda\in\mathbb P^1(k)$.
  Consider the torsion class $\calR\cup\calI$.
 It has a canonical join representation $\Join\limits_{\lambda\in \mathbb P^1(k)}\{\calT_{R_\lambda}\}$ and hence  $\calR\cup\calI\in\tors_0(\Lambda)\setminus \ftors(\Lambda)$.
 On the other hand, consider the torsion class $\calI$.
 It has a canonical join representation which is $\Join\{\calI\}$ (join of itself).
 Note that $\calI$ is join irreducible but not completely join irreducible. 
 Therefore $\calI\in \{\calT\in\tors(\Lambda)| \CJR(\calT) \text{\ exists}\}\setminus \tors_0(\Lambda)$. 
One can check that $\bar{\kappa}(\calI\cup\calR)=\calI$. 
However, since $\calI$ is not completely join-irreducible, $\bar\kappa(\calI)$ is not defined. Hence $\bar\kappa^2(\calR\cup\calI)$ is not defined.
\end{example}

\section{Proof of Theorem~\ref{main}}\label{main sec}
In this section, we first recall some facts about wide subcategories.
Then we prove Theorem \ref{main} as Theorem~\ref{k and d} in subsection~\ref{Theorem D}.  
\subsection{Wide Subcategories}\label{wide}
In this section we recall the definition of a wide subcategory, and introduce two important maps between $\wide \Lambda$ and $\tors \Lambda$.
The main result is Proposition~\ref{simple=coext} which relates the simple objects of a wide subcategory and minimal (co)extending modules of an associated torsion class. 

Let $\calA$ be an abelian category. A subcategory $\calD$ is an \emph{abelian subcategory} of $\calA$ if $\calD$ is abelian and the embedding functor $i:\calD\to \calA$ is exact. An abelian subcategory $\calW$ of $\calA$ is called \emph{wide} if $\calW$ is closed under extensions. It is well-known that the following are equivalent:

\begin{enumerate}
\item $\calW$ is a wide subcategory.
\item $\calW$ is closed under extensions, kernels and cokernels.
\item $\calW$ satisfies the two-of-three property. i.e. Let $0\to A\to B\to C\to 0$ be an exact sequence in $\calA$. If two of three objects $A, B, C$ belong to $\calW$, then so does the third one.
\end{enumerate}

\begin{remark}
\normalfont
Wide subcategories are  also referred to as \emph{thick subcategories}. 
\end{remark}

The relation between wide subcategories and torsion classes is studied in \cite{MS}. Let $\Lambda$ be an artin algebra. Denote by $\wide \Lambda$ the class of wide subcategories of $\module\Lambda$. We define maps between $\wide \Lambda$ and $\tors \Lambda$ as in \cite{MS}:
\begin{eqnarray*}
\alpha: \tors \Lambda&\to& \wide \Lambda\\
\calT &\mapsto& \{X\in \calT: \forall (g:Y\to X)\in \calT,\, \ker(g)\in \calT\}\\
\beta:  \wide \Lambda &\to&\tors \Lambda\\
\calW &\mapsto& \calT_{\calW} = \filt(\Gen(\calW)).
\end{eqnarray*}

Dually, one can define maps:
\begin{eqnarray*}
\alpha': \torf(\Lambda)&\to& \wide \Lambda\\
\calF &\mapsto& \{X\in \F: \forall (g:X\to Y)\in \F,\, \cok(g)\in \F\}\\
\beta':  \wide \Lambda &\to&\torf(\Lambda)\\
\calW &\mapsto& \F_{\calW} = \filt(\Cogen(\calW)).
\end{eqnarray*}

\begin{prop}\cite[Proposition 3.3]{MS} \label{wide 1}
Let $\calW$ be a wide subcategory in $\module\Lambda$. Then $\alpha\beta(\calW)=\calW$ and $\alpha'\beta'(\calW)=\calW$.
\end{prop}

For a general algebra $\Lambda$, the maps $\alpha$ and $\beta$ are \emph{not} mutually inverse bijections.
By restricting the domain of $\alpha$ to functorially finite 
torsion classes 
$\ftors\Lambda$ and the domain of $\beta$ to certain special subset of $\fwide\Lambda$, where $\fwide\Lambda$ denotes functorially finite wide subcategories,  Marks and  \v{S}\v{t}ov\'{\i}\v{c}ek obtain a bijection.

\begin{proposition}\cite[Theorem 3.10]{MS}\label{ftor=fwide}
Let $\Lambda$ be an artin algebra. 
 Then $\alpha$ yields a bijective correspondence between\\
 (1) functorially finite torsion classes\\
 (2) functorially finite wide subcategories $W$ such that $\beta(W)$ is functorially finite.
\end{proposition}

For an arbitrary artin algebra $\Lambda$, there are injections (but not inclusions) of the following sets:
$$
\ftors\Lambda
\mathrel{\mathop{\hookrightarrow}^{\alpha}} 
\fwide\Lambda\hookrightarrow\wide\Lambda\mathrel{\mathop{\hookrightarrow}^{\beta}}  \tors\Lambda.
$$
There are situations where all these can be proper injections (see \cite[3.4, 4.7]{MS}).
We further study these injections by restricting to $\tau$-tilting finite algebras and to hereditary algebras.

\begin{corollary}
Let $\Lambda$ be a $\tau$-tilting finite algebra.
Then the map $\alpha: \ftors \Lambda \to \fwide\Lambda$ is a bijection.
\end{corollary}
\begin{proof}
When $\Lambda$ is a $\tau$-tilting finite algebra, $\ftors\Lambda=\tors\Lambda$ is a finite set. So the composition of three injections above is an injection of a finite set to itself. Hence $\alpha$ is bijective.
\end{proof}
\noindent The next Theorem, due to Ingalls-Thomas, will be useful in our proof of Theorem~\ref{main}.

\begin{theorem}\cite[Corollary 2.17]{IT}\label{IT bij}
Let $H$ be a hereditary artin algebra.
Then  $\alpha: \ftors H \to \fwide H$ is a bijection.
\end{theorem}

\begin{remark}
\normalfont
The result in \cite{IT} stated as  $\alpha$ is a bijection between \emph{finitely generated} torsion classes and  \emph{finitely generated}  wide subcategories. We need to justify the equivalence of terminologies here.

 Recall that a torsion class $\calT$ is called \emph{finitely generated} if $\calT=\Gen(M)$ for some finitely generated module $M$. A wide subcategory $\calW$ is called \emph{finitely generated} if it is finitely generated as an abelian category. 
Finitely generated wide subcategories are also called \emph{exceptional subcategories}  in \cite[\S 2.2.2]{R}. 
 
 It is well known that a torsion class is finitely generated if and only if it is functorially finite. We claim that for a hereditary algebra $H$, finitely generated wide subcategories are exactly functorially finite wide subcategories.
In fact, if $\calW\in\wide H$ is finitely generated, then $\beta(\calW)=\Gen\calW$ \cite[2.13]{IT} is a finitely generated torsion class. Hence by Proposition \ref{ftor=fwide}, $W=\alpha\beta(W)$ is functorially finite. Conversely, if $\calW\in\fwide H$,  then $\calW$ is generated by the minimal left $\calW$-approximation of the projective generator $H$. Hence it is finitely generated. 
  \end{remark}

The following result is due to Ringel: 
\begin{prop}\cite{R2, R4} \label{wide 3}
Let $\{M_i\}_{i\in I}$ be a set of $\Hom$-orthogonal bricks. Then $\filt(\{M_i\}_{i\in I})$ is a wide subcategory of $\module\Lambda$, where the $M_i$'s are the simple objects of $\filt(\{M_i\}_{i\in I})$.

Conversely, if $\calW$ is a wide subcategory in $\module\Lambda$ then the set of simple objects $\{M_i\}_{i\in I}$ in $\calW$ are $\Hom$-orthogonal bricks and $\calW= \filt(\{M_i\})$.
\end{prop}

Recall from Theorem~\ref{torsion_cjr} that each canonical join representation $\CJR(\calT) = \Join_{i\in I} \{\calT_{M_i}\}$ is also determined by a collection $\{M_i\}_{i\in I}$ of hom-orthogonal bricks.
We make this connection explicit below.

\begin{corollary}\label{CJC and simple objects}
Let $\Lambda$ be a finite dimensional algebra and let $\calT$ be a torsion class in $\tors \Lambda$ with canonical join representation $\CJR(\calT) = \Join_{i\in I} \{\calT_{M_i}\}$.
Then the set of bricks $\{M_i\}_{i\in I}$ is the set of simple objects in $\alpha(\calT)$.
\end{corollary}

\begin{proof}
Proposition~\ref{wide 3} says that $\calW = \filt(\{M_i\}_{i\in I})$  is a wide subcategory with $\{M_i\}_{i\in I}$ the set of simple objects. 
Observe that 
\begin{align*}
\beta(\calW) =& \filt( \Gen(M_i))_{i\in I} \\[.5em]
=& \Join_{i\in I}  \filt(\Gen( M_i))\\
=& \Join_{i\in I} \calT_{M_i}\\[.5em]
=&\calT
\end{align*}
By Proposition~\ref{wide 1}, $\alpha(\calT) = \calW$.
The statement follows. 
\end{proof}

\subsection{Simple Objects in Wide Subcategories}

We now describe the simple objects of wide subcategories in terms of minimal extending modules.
This proposition is key to our proof of Proposition~\ref{wide many perp} and Theorem~\ref{main}.
\begin{prop}\label{simple=coext}
Let $\Lambda$ be an artin algebra and $\calT\in\tors\Lambda$. Let $\alpha(\calT)$ be the corresponding wide subcategory and $\calF=\calT^\perp$ the torsion-free class. Then $M$ is a simple object in $\alpha(\calT)$ if and only if $M$ is a minimal co-extending module of $\calF$.
\end{prop}
\begin{proof} ``if part'': Let $M$ be a minimal co-extending module of $\calF$. Then since $\Hom(M,\calF)=0$, we have $M\in\calT$. Let $f:T\to M$ be a homomorphism in $\calT$. We want to show that $K=\ker f\in\calT$. In fact, if $f$ is not an epimorphism then $\image f$ is a proper submodule of $M$. Hence $\image f\in\calF\cap\calT=0$ which implies $\ker f\cong T\in\calT$. Now assume $f$ is an epimorphism.  Take the canonical sequence  $0\to tK\to K\to K/tK\to 0$ of $K$ with respect to the torsion pair $(\calT,\calF)$. Then in the push-out diagram we have:
$$
\begin{tikzcd}
0 \arrow[r] & K \arrow[r, "i"] \arrow[d, "p\quad (p.o.)"]& T \arrow[r, "f"]\arrow[d, "h"] & M\arrow[r]\arrow[d, equal]& 0 \\
0 \arrow[r] & K/tK \arrow[r] & E \arrow[r] & M \arrow[r]& 0
\end{tikzcd}
$$
If the lower exact sequence is not split, then $E\in\calF$ by Property~\ref{propertytwoprime} of minimal co-extending modules.
Then $h=0$ and hence $p=0$. 
If the lower exact sequence split, then $p$ factors through $i$ and therefore $p=0$ again.
So in both cases $K\cong tK\in\calT$. Hence $M\in\alpha(\calT)$.
All that remains is to show that $M$ is simple in $\alpha(\calT)$.
Let $M'\in\alpha(\calT)$ be a proper submodule of $M$.
 Then $M'\in\calF\cap\calT=0$. 
Therefore $M$ is simple in $\alpha(\calT)$. 

``only if part'': Let $M$ be a simple object in $\alpha(\calT)$. First, since $M\in\calT$, $\Hom(M,\calF)=0$. This is Property~\ref{propertythreeprime}.
Second, let $M'$ be a proper submodule of $M$.
Then the torsion part $tM'$ is also a proper submodule of $M$.
Therefore, for any $g:X\to tM'$ in $\calT$, we have $\ker(g)=\ker(i\circ g)$ is in $\calT$, where $i:tM' \hookrightarrow M$.
 It follows that $tM'\in\alpha(\calT)$.
 Since $M$ is simple, $tM'=0$.
 Therefore $M'\in\calF$.
 This is Property~\ref{propertyoneprime}
 Last let $0\to F\to E\to M\to 0$ be a non-split exact sequence with $F\in\calF$.
We need to show $E\in\calF$.
Take the canonical sequence  $0\to tE\to E\to E/tE\to 0$ of $E$ with respect to the torsion pair $(\calT,\calF)$. In the pull-back diagram:
$$
\begin{tikzcd}
0 \arrow[r] & L \arrow[r] \arrow[d, hook, " \quad (p.b.)"]& tE \arrow[r]\arrow[d, hook] & N\arrow[r]\arrow[d,hook, "\iota"]& 0 \\
0 \arrow[r] & F \arrow[r] & E \arrow[r] & M \arrow[r]& 0
\end{tikzcd}
$$
Notice that $L\in\calF$.
 If $\iota$ is a proper monomorphism, then we have shown that $N\in\calF$. Hence $tE\in\calF\cap\calT=0$ and $E\cong E/tE\in\calF$ which finishes the proof. If $\iota$ is an isomorphism, then since $M\in\alpha(\calT)$ we get $L\in\calT\cap\calF=0$. Hence $tE\cong N\cong M$. But then the commutative diagram implies that the lower exact sequence is split, a contradiction. 
 Therefore $M$ satisfies Property~\ref{propertytwoprime}.
 We conclude that $M$ is a minimal coextending module of $\calF$.
\end{proof}

\subsection{Ringel's $\epsilon$-map}\label{etamap}

In \cite{R}, Ringel studied the perpendicular categories of wide subcategories over a hereditary algebra and its combinatorial connection with the Kreweras complement from Remark~\ref{Kreweras}. In this section, we will show that our map $\bar\kappa$ (defined on torsion classes) is the counterpart of taking perpendicular categories of wide subcategories under the Ingalls-Thomas bijection (see Theorem \ref{k and d}). 

Let $\calW$ be a wide subcategory of $\module H$ for a hereditary algebra $H$. Then it is easy to see  that the subcategory $\leftidx{^{\perp_{0,1}}}\calW=\{X\in\module H | \Ext^i_H(X,M)=0, \forall M\in\calW, i=0,1\}$ is also a wide subcategory. Similarly the subcategory $\calW^{\perp_{0,1}}$ is again a wide subcategory.

\begin{definition}\label{delta def}
Let $H$ be a hereditary artin algebra and $\calW$ be a wide subcategory in $\module H$. Define  $\delta(\calW)=\calW^{\perp_{0,1}}$ and $\epsilon(\calW)=\leftidx{^{\perp_{0,1}}}\calW$ .
\end{definition}

 \begin{definition}
Let $H$ be a hereditary artin algebra and $\calU$, $\calV$  subcategories in $\module H$. Then $(\calU,\calV)$ is called a perpendicular pair provided $\calU=\leftidx{^{\perp_{0,1}}}\calV$ and  $\calV= \calU^{\perp_{0,1}}$. 
\end{definition}

\begin{lemma}\label{perp pair}
Let $H$ be a hereditary artin algebra and $(\calU,\calV)$ a perpendicular pair.
\begin{itemize}
\item \cite{R} Then $\calU$ is a functorially finite wide subcategory if and only if $\calV$ is a functorially finite wide subcategory.
\item Therefore $\epsilon$ is a bijection from $\fwide H$ to $\fwide H$.
\end{itemize}
\end{lemma}

In the following, we will heavily use the Auslander-Reiten formulas in the special case when $H$ is a hereditary algebra \cite[Proposition 3.4]{A}\cite[IV 2.13, 2.14]{ASS}. 

\begin{theorem}[Auslander-Reiten formulas]\label{AR_formulas}
Let $H$ be hereditary.
 Then there exist isomorphisms that are functorial in both variables:
$$
\Ext^1_H(M,N)\cong D {\Hom_H(\tau^{-1}N,M)}\cong D {\Hom_H(N,\tau M)}.
$$
\end{theorem}

As an immediate application we have:

\begin{lemma} \label{hereditary brick}
Suppose that $H$ is a hereditary algebra and $M$ is not injective.
Then $M$ is a brick if and only if $\tau^{-1}M$ is a brick.
\end{lemma}
\begin{proof}
Using the AR-formulas, we have:
$$
\Hom_{\Lambda}(\tau^{-1}M, \tau^{-1}M)\cong D\Ext_{\Lambda}^1(\tau^{-1}M,M)\cong DD\Hom_{\Lambda}(M,\tau\tau^{-1}M)\cong \Hom_{\Lambda}(M,M).
$$
So $M$ is a brick if and only if so is $\tau^{-1}M$.
\end{proof}

\begin{remark}
\normalfont
Lemma \ref{hereditary brick} does not hold if $\Lambda$ is not hereditary. For example, let $\Lambda$ be the preprojective algebra of the quiver
$$
4\to3\to2\to1.
$$
 One can check that $\tau^{-1}\left(\begin{smallmatrix} 3\\4\end{smallmatrix}\right)=\begin{smallmatrix} &2\\3&&1\\&2\end{smallmatrix}$, where $\begin{smallmatrix} 3\\4\end{smallmatrix}$ is a brick and $\begin{smallmatrix} &2\\3&&1\\&2\end{smallmatrix}$ is not.
\end{remark}

There is another well-known result by Auslander and Smal\o, which is obtained from the AR-formulas.
\begin{lemma} \cite[Proposition 5.6, 5.8]{AS}\label{ARF var}
Let $\Lambda$ be an artin algebra and $M, N$ be $\Lambda$-modules. Then\\
$(a)$ $\Hom_{\Lambda}(\tau^{-1} M, N)=0$ if and only if $\Ext^1_{\Lambda}(\Cogen(N), M)=0$.\\
$(b)$ $\Hom_{\Lambda}(M,\tau N)=0$ if and only if $\Ext^1_{\Lambda}(N, \Gen(M))=0$.
\end{lemma}

As an application, we will use Lemma~\ref{ARF var} to compute $\alpha(\leftidx{^\perp}{M})$ for a brick module $M$. Notice that $\alpha(\leftidx{^\perp}{M})$ is just the corresponding wide subcategory of the torsion class $\kappa(\calT_M)$ under the Ingalls-Thomas bijection from Theorem~\ref{IT bij}. (See Remark \ref{wide perp rmk} and Corollary \ref{wide perp cor} below.)

\begin{lemma}\label{wide perp}
Let $\Lambda$ be a finite dimensional algebra and let $M$ be a brick in $\module \Lambda$.
Then $\alpha(\leftidx{^\perp}{M})=\{X\in \leftidx{^\perp}{M} | \forall X'\in\Cogen X\cap \leftidx{^\perp}{M}, \Ext^1(X', M)=0\}$.
\end{lemma}

\begin{proof}
For ``$\subseteq$'': Let $X\in\alpha(\leftidx{^\perp}{M})$ and $X'\in \Cogen X\cap \leftidx{^\perp}{M}$. Suppose there is an non-split exact sequence $0\to M\to E\stackrel{p}\to X'\to 0$. Since $M\in \ME(\leftidx{^\perp}{M})$ by Proposition \ref{left perp}, it follows that $E\in \leftidx{^\perp}{M}$. But then the composition $f:E\stackrel{p}\to X' \hookrightarrow \oplus X$ is a homomorphism in $\leftidx{^\perp}{M}$ with $\Ker f=M\not\in \leftidx{^\perp}{M}$, which is a contradiction with $X\in\alpha(\leftidx{^\perp}{M})$.

For ``$\supseteq$'': Let $X\in\{X\in \leftidx{^\perp}{M} |  \forall X'\in\Cogen X\cap \leftidx{^\perp}{M}, \Ext^1(X', M)=0\}$ and $f:Y\to X$ any homomorphism in $\leftidx{^\perp}{M}$, we want to show that $\Ker f$ is also in $\leftidx{^\perp}{M}$.

Since $Y\in \leftidx{^\perp}{M}$, so is $\image f$. So $\Ext^1(\image f,M)=0$ by the assumption. Then in the exact sequence:
$$
0\to \Hom(\image f, M)\to \Hom(Y,M)\to \Hom(\Ker f,M)\to \Ext^1(\image f, M),
$$
both the terms $\Hom(Y,M)$ and $\Ext^1(\image f, M)$ vanishes. Hence $\Ker f\in \leftidx{^\perp}{M}$.
\end{proof}

\begin{remark}\label{wide perp rmk}
\normalfont
Due to Lemma \ref{ARF var}, for any $\Lambda$-module $M$, the subcategory $$\leftidx{^\perp}{M}\cap(\tau^{-1} M)^\perp=\{X\in\leftidx{^\perp}{M}| \forall X'\in\Cogen X, \Ext^1(X', M)=0\}.$$ 
Therefore there are inclusions of subcategories:
$$
\leftidx{^\perp}{M}\cap(\tau^{-1} M)^\perp \subseteq  \alpha(\leftidx{^\perp}{M}) \subseteq \leftidx{^{\perp_{0,1}}}{M}.
$$
\end{remark}

\begin{cor}\label{wide perp cor}
Let $H$ be a hereditary artin algebra and  $M$ be a brick in $\module H$.
Then $\alpha(\leftidx{^\perp}{M})= \leftidx{^{\perp_{0,1}}}{M}$.
\end{cor}
\begin{proof}
When $H$ is hereditary $(\tau^{-1}M)^\perp= \leftidx{^{\perp_{1}}}M$ by Theorem~\ref{AR_formulas}.
Therefore \[\leftidx{^\perp}{M}\cap(\tau^{-1} M)^\perp= \leftidx{^\perp}{M}\cap\leftidx{^{\perp_{1}}}M = \leftidx{^{\perp_{0,1}}}{M}.\]
 The assertion follows Remark \ref{wide perp rmk}.
\end{proof}

Questions about whether both inclusions in Remark \ref{wide perp rmk} can be proper and if $\leftidx{^\perp}{M}\cap(\tau^{-1} M)^\perp$ is always a wide subcategory for arbitrary finite dimensional algebra are answered later in Appendix. Following our main task, we prove a generalized version of Corollary \ref{wide perp cor}.

\begin{prop}\label{wide many perp}
Let $H$ be a  finite dimensional hereditary algebra and $\{M_i\}$ a set of $\Hom$-orthogonal bricks. Then $\alpha(\bigcap\leftidx{^\perp}M_i)=\bigcap\alpha(\leftidx{^\perp}M_i)$.
\end{prop}
\begin{proof}
It suffices to show that $\alpha(\bigcap\leftidx{^\perp}M_i)\subseteq\bigcap\alpha(\leftidx{^\perp}M_i)$. By Corollary \ref{wide perp cor}, $\bigcap\alpha(\leftidx{^\perp}M_i)=\bigcap \leftidx{^{\perp_{0,1}}}M_i$ which is closed under extensions. So it suffices to show that each simple object in $\alpha(\bigcap\leftidx{^\perp}M_i)$ lies in $\leftidx{^{\perp_{0,1}}}M_i$ for all $i$.

Assume $X$ is a simple object in  $\alpha(\bigcap\leftidx{^\perp}M_i)$. By Proposition \ref{simple=coext}, $X$ is a minimal co-extending module of $\calF:=(\bigcap\leftidx{^\perp}M_i)^\perp=\Join\filt\Cogen M_i$. Since $\Hom_H(X,\calF)=0$, it follows that $X\in\leftidx{^\perp}M_i$ for all $i$. 

So it remains to show that $\Ext^1_H(X,M_i)=0$ for all $i$. Assume there is a non-split exact sequence $0\to M_i\stackrel{l}\to E\stackrel{p}\to X\to 0$. Then by the assumption that $X$ is a minimal co-extending module, we known $E\in\calF$. Notice that $\Join\filt\Cogen M_i=\filt(\bigcup \Cogen M_i)$. There is a non-zero epimorphism $\pi:E\to E_0$, where $E_0\in\Cogen M_j$ for some $j$. We claim that $\pi\circ l\neq 0$. Otherwise, $\pi$ can factor through $X$. But then $\pi=0$ since $\Hom_H(X,\calF)=0$, a contradiction. So the composition $M_i\stackrel{l}\to E\stackrel{\pi}\to E_0\to\oplus M_j$ is non-zero. Since $\{M_i\}$ is $\Hom$-orthogonal, we have $i=j$. Since $M_i$ is a brick, the composition is indeed a split monomorphism. Hence $l$ is a split monomorphism which is a contradiction.
\end{proof}

\subsection{Theorem D: $\bar\kappa$ and $\epsilon$}\label{Theorem D}
\begin{theorem}[Theorem \ref{main}]\label{k and d} 
Let $H$ be a hereditary artin algebra. 

For torsion classes $\calT\in\tors_0\Lambda$, $\alpha\bar\kappa(\calT)= \epsilon\alpha(\calT)$, i.e. the following diagram commutes:
\begin{center}
\begin{tikzpicture}[scale=0.60]
  \node (A) at (0,0) {$\tors_0 H$}; 
  \node(B) at (4,0) {$\tors H$};
  \node (C) at (0,-3) {$\wide H$}; 
  \node (D) at (4,-3) {$\wide H.$};
    \node at (2,-2.75) {$\epsilon$};
    \node at (2,.25) {$\bar{\kappa}$};
    \node at (-.25,-1.5) {$\alpha$};
    \node at (4.25,-1.5) {$\alpha$};
  \draw[->] (A.east)--(B.west);
  \draw[->] (A.south)--(C.north);
  \draw[->] (C.east)--(D.west);
  \draw[->] (B.south)--(D.north);
\end{tikzpicture}
\end{center}
\end{theorem}
\begin{proof}
Assume $\CJR(\calT)=\Join_{i\in I}\calT_{M_i}$. Then
\begin{eqnarray*}
\alpha\bar\kappa(\calT)&\overset{Cor.~\ref{kappa tors}}{=\joinrel=}&\alpha(\bigcap\limits_{i\in I}\leftidx{^\perp}M_i)\\
&\overset{Prop.~\ref{wide many perp}}{=\joinrel=}&\bigcap\limits_{i\in I}\alpha(\leftidx{^\perp}M_i)\\
&\overset{Cor.~\ref{wide perp cor}}{=\joinrel=}&\bigcap\limits_{i\in I}\leftidx{^{\perp_{0,1}}}{M_i}.
\end{eqnarray*}
On the other hand, 
\begin{eqnarray*} 
\epsilon \alpha(\calT)&\overset{Prop.~\ref{wide 3}}{=\joinrel=}&\epsilon\filt(\{M_i\}_{i\in I})\\
&\overset{Def. \ref{delta def}}{=\joinrel=}&\leftidx{^{\perp_{0,1}}}\filt(\{M_i\}_{i\in I})
\end{eqnarray*}

Comparing these results, it is obvious that $\leftidx{^{\perp_{0,1}}}\filt(\{M_i\}_{i\in I})\subseteq\bigcap\limits_{i\in I}\leftidx{^{\perp_{0,1}}}{M_i}$. Conversely, pick any $X\in\bigcap\limits_{i\in I}\leftidx{^{\perp_{0,1}}}{M_i}$ and any $Y\in\filt(\{M_i\}_{i\in I})$, one can check that both $\Hom(X,Y)=0$ and $\Ext^1(X,Y)=0$. So $\bigcap\limits_{i\in I}\leftidx{^{\perp_{0,1}}}{M_i}\subseteq\leftidx{^{\perp_{0,1}}}\filt(\{M_i\}_{i\in I})$. Therefore, $\alpha\bar\kappa(\calT)= \epsilon\alpha(\calT).$
\end{proof}
 
 \begin{remark}
 \normalfont
 For wide subcategories $\calW$, $\bar\kappa\beta(\calW)=\beta\epsilon(\calW)$ is not always true. For example, in Example \ref{Kronecker torsions}, let $\calW=\calR$. One can check that $\bar\kappa\beta(\calW)=\calI$, whereas $\beta\epsilon(\calW)=0$. 
 \end{remark}
 
Theorem \ref{k and d} shows another surprising fact that $\bar\kappa$ is a bijection of the set of functorially finite torsion classes:  

\begin{corollary}\label{k ftors}
Let $H$ be a hereditary artin algebra and $\calT$ a functorially finite torsion class. Then $\bar\kappa(\calT)$ is functorially finite.
\end{corollary}
\begin{proof}
By Theorem~\ref{IT bij} and Lemma \ref{perp pair}, the restriction of $\bar{\kappa}$ to $\ftors H$ is a composition of three bijections, $\alpha \epsilon \alpha^{-1}: \ftors H \to \ftors H$.
\end{proof}
 
\begin{example}
\normalfont
If $\Lambda$ is not hereditary, $\bar\kappa$ does not necessarily preserve functorial finiteness. We show an example which is the dual of \cite[Example~4.5]{A}. Consider the algebra $\Lambda=kQ/\langle\alpha\gamma\rangle$ with the quiver
$$
\begin{tikzcd}
Q: 1&2\arrow[l,yshift=0.7ex,swap,"\alpha"]\arrow[l,yshift=-0.7ex,"\beta"]&3\arrow[l,swap,"\gamma"].
\end{tikzcd}
$$
Take the brick $P_3$. Since $P_3$ is projective, $\Gen(P_3)$ is a functorially finite torsion class. Hence $\calT=\filt\Gen(P_3)=\Gen(P_3)=\add\{P_3, I_2, I_3\}$ is functorially finite. On the other hand, $\kappa(\calT)=\leftidx{^\perp}P_3$. Let $R_\alpha=\begin{tikzcd}
k&k\arrow[l,yshift=0.7ex,swap,"\alpha"]\arrow[l,yshift=-0.7ex,"1"]&0\arrow[l,swap,"0"]\end{tikzcd}$.
It can be directly computed that $\leftidx{^\perp}P_3=\Join_{\alpha\neq0}\filt(\Gen R_\alpha)\Join\filt(\Gen(S_3))$, which is not covariantly finite.

 It can also be shown using a result of Smal\o\ from \cite{S} that $\leftidx{^\perp}P_3$ is functorially finite if and only if $(\leftidx{^\perp}P_3)^\perp=\filt\Cogen P_3$ is functorially finite.
 A result of Asai from \cite{A} shows that $\filt\Cogen P_3$ is not functorially finite.
\end{example}

 \section{Iterated Compositions of $\bar{\kappa}$ and Relation with AR translation}\label{sec:AR}
 
 Let $H$ be a hereditary algebra and $\delta$, $\epsilon$ be the maps between wide subcategories defined before. Ringel \cite{R} describes the map $\delta^2$ using the so-called \emph{extended Auslander-Reiten translation}:
 
 \begin{definition}
 Define $\bar\tau$ to be a permutation of indecomposable $H$-modules as follows:
 $$
 \bar\tau X=
 \begin{cases}
 \tau X & X \text{\ not projective}\\
 I(S) & X=P(S)
 \end{cases}
 $$
 where $I(S)$ and $P(S)$ are the injective envelope and projective cover of the same simple $S$.
 \end{definition}
 
 \begin{theorem}
 Let $H$ be a hereditary algebra and $\calW\in\fwide\Lambda$. Then $\delta^2\calW=\bar\tau\calW$.
 \end{theorem}
 
 Here, we give a similar result for torsion classes using our $\bar\kappa$ map.
 
 \begin{cor} \label{ar noninj}
 Let $H$ be a hereditary algebra and $\calT\in\ftors\Lambda$ with a canonical join representation $\Join\limits_{1\leq i\leq k}\filt\Gen(M_i)$. Then ${\bar\kappa}^2(\calT)$ is a functorially finite torsion class.
 Furthermore, if each $M_i$ is non-injective, then $\bar{\kappa}^2(\calT)$ has a canonical join representation $\Join\limits_{1\leq i\leq k}\filt\Gen({\tau}^{-1}M_i)$. 
 \end{cor}
 \begin{proof}
 The functorially finiteness of  ${\bar\kappa}^2(\calT)$ follows from Corollary \ref{k ftors}.
 Now assume that each $M_i$ is non-injective.
 Then according to Theorem \ref{k and d} $\alpha{\bar\kappa}^2(\calT)=\epsilon^2\alpha(\calT)=\epsilon^2\filt(\{M_i\}_{1\leq i\leq k})=\bar{\tau}^{-1}\filt(\{M_i\}_{1\leq i\leq k})$.
 
 Since each $M_i$ is non-injective, $\filt(\{M_i\}_{1\leq i\leq k})$ ($=\filt(\{M_i\})$ for short below) contains no injective modules.
 Therefore $\bar{\tau}^{-1}(X) =\tau^{-1}(X)$ for each $X\in \filt(\{M_i\})$.
Therefore $\tau^{-1}(\filt(\{M_i\}) = \bar{\tau}^{-1}(\filt(\{M_i\})$.
Now $\tau^{-1}$ is an equivalence from $\filt(\{M_i\})$ to $\tau^{-1}(\filt(\{M_i\})$ (because $H$ is hereditary and $\filt(\{M_i\})$ contains no injective modules).
Hence $\tau^{-1}$ preserves simple objects.
So $\bar{\tau}^{-1}\filt(\{M_i\})$ is equal to $\filt(\{\tau^{-1}M_i\})$.
  Using Theorem~\ref{IT bij}, ${\bar\kappa}^2(\calT)=\Join\limits_{1\leq i\leq k}\filt\Gen({\tau}^{-1}M_i)$.
 \end{proof}
 
 \begin{remark} 
\normalfont
  Let $H$ be a hereditary algebra and $I$ be an indecomposable injective module. Then for the functorially finite torsion class $\calT=\filt\Gen(I)=\Gen(I)$, ${\bar\kappa}^2\calT=\filt\Gen(\bar{\tau}^{-1}I)$. However, Corollary \ref{ar noninj} does not always hold if some $M_i$'s are injective:
 
 Let $\Lambda=kQ$ be the path algebra with $Q=3\to2\to1$. \\Take $\calT=\filt\Gen(S_2)\Join\filt\Gen(S_3)$, then ${\bar\kappa}^2(\calT)=\filt\Gen(P_2)\Join\filt\Gen(S_3)$, where $P_2\neq \bar{\tau}^{-1} S_2,  \bar{\tau}^{-1} S_3$.
 \end{remark}
 
 We conclude this section with a proof of Theorem~\ref{Kr_theorem}, which we restate below.
 \begin{theorem}[Theorem~\ref{Kr_theorem}]
Let $\Lambda$ be $\tau$-tilting finite, and let $r$ be the number of simples in $\module\Lambda$.
For any $\calT\in \tors \Lambda$ let $|\calT|$ denote the number of canonical joinands of $\calT$.
Then for any $\bar{\kappa}$-orbit $\calO$ we have
\[\frac{1}{|\calO|} \sum_{\calT\in \calO}|\calT| = r/2.\]
 \end{theorem}
 
 \begin{proof}
Since $\Lambda$ is $\tau$-tilting finite, the lattice $\tors \Lambda$ is finite and \emph{regular} meaning that for any $\calT\in \tors\Lambda$:
\[|\{\calS: \calT \covers \calS\}|+ |\{\calU:\calT\covered \calU\}| = r\]
For the remainder of the proof, fix $\calT$ in $\tors \Lambda$, and write $\calO_\calT$ for the $\bar{\kappa}$-orbit of $\calT$.

Suppose that $|\calT|=k$ where $r-k \ne k$.
Observe that if $\bar{\kappa}(\calT')= \calT$ then $|\calT'|=r-k$, because $\tors \Lambda$ is regular. 
Also, $|\bar{\kappa}^i(\calT)|=r-k$ if and only if $i$ is odd.
Therefore, $|\calO_\calT|= (i+1)$, and the sum $\sum_{\calS\in \calO_\calT}|\calS|$ consists of $(i+1)/2$ pairs of torsion classes $\calS$ and $\bar{\kappa}(\calS)$ with $|\calS|+|\bar{\kappa}(\calS)|=k+(r-k)$.
We compute:
\begin{align*}
\frac{1}{|\calO_\calT|} \sum_{\calS\in \calO_\calT}|\calS| = & \frac{1}{(i+1)}\left(k+(r-k)+k+(r-k)+\cdots + k+(r-k)\right)\\
&=\frac{1}{(i+1)}\left(k+(r-k)\right)\frac{(i+1)}{2}\\
&=\frac{r}{2}
\end{align*}
When $r$ is even and $|\calT|=r/2$ then we have:
\begin{align*}
\frac{1}{|\calO_\calT|} \sum_{\calS\in \calO_\calT}|\calS| =&
\frac{1}{|\calO_\calT|}\left(\frac{r}{2}\right)|\calO_\calT|\\
=&\frac{r}{2}
\end{align*}
\end{proof}
 \section{Appendix}
We show some minor results which we obtained during our study of the perpendicular categories in previous sections.  

\begin{example}
\normalfont
When $\Lambda$ is not hereditary, both inclusions in Remark \ref{wide perp rmk} can be proper: 

Let $Q$ be the $\mathbb A_2$ quiver and $\Lambda=kQ\otimes (k[x]/\langle x^2\rangle)$. The algebra $\Lambda$ is given by the following quiver with relations:
$$
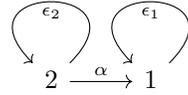
\begin{tikzcd}
 2\arrow[loop, left, "\epsilon_2"]\ar[r, "\alpha"] &1\arrow[loop,right, "\epsilon_1"]
 \end{tikzcd}
$$

$$
\epsilon_1^2=\epsilon_2^2=0, \alpha\epsilon_2=\epsilon_1\alpha.
$$
The AR-quiver of $\module\Lambda$ is as the following:

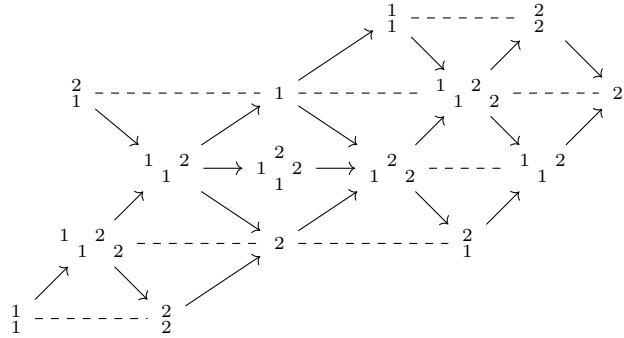
\begin{center}
\begin{tikzpicture}[->]
\node (11a) at (-.5,0) {$\begin{smallmatrix}1\\1\end{smallmatrix}$}; 
\node (1212a) at (.5,1) {$\begin{smallmatrix} 1&&2&\\ &1&&2\end{smallmatrix}$};
\node (22a) at (1.5,0) {$\begin{smallmatrix}2\\2\end{smallmatrix}$};
\node  (121a) at (1.5,2) {$\begin{smallmatrix}1&&2\\&1&\end{smallmatrix}$};
\node (21b) at (0.3,3) {$\begin{smallmatrix} 2\\1\end{smallmatrix}$};
\node (1) at (3,3) {$\begin{smallmatrix}1\end{smallmatrix}$};
\node  (p2) at (3,2) {$\begin{smallmatrix}&2&\\1&&2\\&1&\end{smallmatrix}$};
 \node  (2a) at (3,1) {$\begin{smallmatrix}2\end{smallmatrix}$};
 \node  (212) at (4.5,2) {$\begin{smallmatrix}&2\\1&&2\end{smallmatrix}$};
 \node  (11b) at (4.5,4) {$\begin{smallmatrix}1\\1\end{smallmatrix}$};
 \node  (1212b) at (5.5,3) {$\begin{smallmatrix}1&&2&\\ &1&&2\end{smallmatrix}$};
 \node  (21) at (5.5,1) {$\begin{smallmatrix}2\\1\end{smallmatrix}$};
 \node  (121b) at (6.5,2) {$\begin{smallmatrix}1&&2\\&1&\end{smallmatrix}$};
 \node  (22b) at (6.5,4) {$\begin{smallmatrix}2\\2&\end{smallmatrix}$};
 \node  (2b) at (7.5,3) {$\begin{smallmatrix}2\end{smallmatrix}$};
 
 \draw (11a)--(1212a);  \draw (1212a)--(22a);  \draw (22a)--(2a);
  \draw (1212a)--(121a);  \draw (121a)--(2a);  \draw (121a)--(p2);  \draw (121a)--(1);
   \draw (1)--(212); \draw (p2)--(212);  \draw (2a)--(212);
    \draw (1)--(11b);  \draw (11b)--(1212b);   \draw (21)--(121b);
 \draw (212)--(21);  \draw (212)--(1212b);
  \draw (1212b)--(121b);   \draw (1212b)--(22b);   \draw (22b)--(2b);  \draw (121b)--(2b);
  \draw(21b)--(121a);

\draw[-][dashed] (11b)--(22b);
\draw[-][dashed] (11a)--(22a);
\draw[-][dashed] (1212a)--(2a)--(21);
\draw[-][dashed] (212)--(121b);
\draw[-][dashed] (21b)--(1)--(1212b)--(2b);
 %
\end{tikzpicture}
\end{center}

Let $M=S_1$, then $\leftidx^{\perp} M\cap(\tau^{-1} M)^\perp=0$, $\alpha(\leftidx^{\perp} M)=\add\{\begin{smallmatrix}2\\1\end{smallmatrix},\begin{smallmatrix}&2&\\1&&2\\&1&\end{smallmatrix}\}=\leftidx^{{\perp_{0,1}}} M$.

For the other proper containment, consider the following algebra: Let $Q$ be the quiver $4\stackrel{\gamma}\to3\stackrel{\beta}\to2\stackrel{\alpha}\to1$. Let $\Lambda$ be the algebra $kQ/<\alpha\beta\gamma>$.

The AR-quiver looks like:
$$
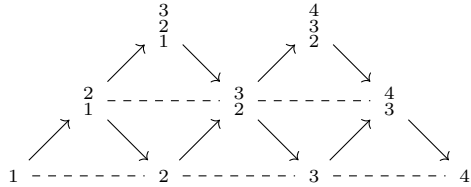
\begin{tikzpicture}[->]
\node  (1) at (0,0)  {$\begin{smallmatrix}1\end{smallmatrix}$}; 
\node (2) at (2,0) {$\begin{smallmatrix}2\end{smallmatrix}$}; 
\node (3) at (4,0) {$\begin{smallmatrix}3\end{smallmatrix}$};
\node (4) at (6,0) {$\begin{smallmatrix}4\end{smallmatrix}$};

\node (21) at (1,1) {$\begin{smallmatrix}2\\1\end{smallmatrix}$};
\node (32) at (3,1) {$\begin{smallmatrix}3\\2\end{smallmatrix}$};
\node (43) at (5,1) {$\begin{smallmatrix}4\\3\end{smallmatrix}$};

\node (321) at (2,2) {$\begin{smallmatrix}3\\2\\1\end{smallmatrix}$};
\node (432) at (4,2) {$\begin{smallmatrix}4\\3\\2\end{smallmatrix}$};

\draw (1)--(21); \draw (21)--(321); \draw (21)--(2); \draw (321)--(32); \draw(2)--(32);
\draw (32)--(3); \draw(32)--(432);\draw(3)--(43); \draw(432)--(43);\draw(43)--(4);   

\draw[-][dashed] (1)--(2)--(3)--(4);
\draw[-][dashed] (21)--(32)--(43);
 \end{tikzpicture}
$$ 

Consider $M=\begin{smallmatrix}2\\1\end{smallmatrix}$, then $\leftidx^{\perp} M\cap(\tau^{-1} M)^\perp=\alpha(\leftidx^{\perp} M)=\add\{\begin{smallmatrix}3\\2\\1\end{smallmatrix},\begin{smallmatrix}2\end{smallmatrix}, \begin{smallmatrix}4\end{smallmatrix}\}$, and $\leftidx^{{\perp_{0,1}}} M=\add\{\begin{smallmatrix}3\\2\\1\end{smallmatrix},\begin{smallmatrix}2\end{smallmatrix}, \begin{smallmatrix}4\\3\\2\end{smallmatrix},  \begin{smallmatrix}4\\3\end{smallmatrix}, \begin{smallmatrix}4\end{smallmatrix}\}$.

 \end{example}

The following result is a generalization of Jasso's \cite[Proposition 3.6]{J},  where we do not assume 
that $M$ is $\tau$-rigid.
\begin{theorem}
Let $\Lambda$ be an artin algebra and let $M$ be any $\Lambda$-module. Then $\leftidx{^\perp}  M \cap (\tau^{-1} M) ^\perp$  is a wide subcategory.
\end{theorem}
\begin{proof}
First, it is clear that  $\leftidx{^\perp}  M \cap (\tau^{-1} M) ^\perp$  is closed under extensions. 

Second, given a morphism $f: Y\to X$ with $X,Y\in \leftidx{^\perp}  M \cap (\tau^{-1} M) ^\perp$, we need to show that both $\Ker f$ and $\cok f\in \leftidx{^\perp}  M \cap (\tau^{-1} M)^\perp$.

Observe that given an exact sequence $0\to A\to B\to C\to 0$. If $B\in\leftidx{^\perp}  M$ then so is $C$. If $B\in (\tau^{-1} M)^\perp$ then so is $A$.  

Using this fact, we known that $\Ker f\in (\tau^{-1} M) ^\perp$, $\cok f\in \leftidx{^\perp}  M$ and $\image f\in \leftidx{^\perp}  M \cap (\tau^{-1} M) ^\perp$.

To show $\Ker f\in \leftidx{^\perp}  M$: Using the exact sequence 
$$
0=\Hom(Y,M)\to\Hom(\Ker f,M)\to\Ext^1(\image f, M).
$$
Since we have already shown $\Hom(\tau^{-1}M, \image f)=0$, using AR formula it follows that 
$$
\Ext^1(\image f, M)\cong D\underline{\Hom}(\tau^{-1}M, \image f)=0.
$$ 
Hence $\Hom(\Ker f,M)=0$.

To show $\cok f\in (\tau^{-1} M) ^\perp$: By Lemma \ref{ARF var} it suffices to show that for any submodule $C'\subseteq (\cok f)^n$, $\Ext^1(C', M)=0$.
In fact, suppose there is an exact sequence $0\to M \stackrel{s} \to E \stackrel{t}\to C'\to 0$  and denote by $i: C'\hookrightarrow (\cok f)^n$ the inclusion. Then there is a pull-back diagram:
$$
\begin{tikzcd}
&&M\arrow[d,hook,"u"']\arrow[r,equal]&M\arrow[d,hook,"s"]\\
0\arrow[r]&(\image f)^n\arrow[d]\arrow[r,"a"] &F\arrow[r,"b"]\arrow[d,"\quad(p.b.)"]\arrow [u,dotted,bend right,"{u'}" description]& E\arrow[r]\arrow[d,"i\circ t"]\arrow[ul,dotted, "c"]&0\\
 0\arrow[r] &(\image f)^n\arrow[r] &X^n\arrow[r] &(\cok f)^n\arrow[r] &0
 \end{tikzcd}
$$

Since $X\in (\tau^{-1} M)^\perp$, $\Ext^1(\Cogen X, M)=0$. Hence $u$ is a split monomorphism. i.e. there exists $u': F\to M$ such that $u'\circ u= 1$. Since $\image f\in \leftidx{^\perp}  M$, the composition $u'\circ a=0$. Hence $u'$ factors through $b$. i.e. there exists $c: E\to M$ such that $u'=c\circ b$.

But since $cs=cbu=u'u=1$, $s$ is a split monomorphism. Hence the exact sequence $0\to M \stackrel{s} \to E \stackrel{t}\to C'\to 0$ splits and therefore $\Ext^1(C',M)=0$.
\end{proof}

\begin{corollary}
\begin{enumerate}
\item Dually  $M^\perp \cap \leftidx{^\perp}\tau M$ is a wide subcategory. 
\item Let $\calM$ be a set of $\Lambda$-modules, then the full subcategory $\{X|X\in M \cap (\tau^{-1} M) ^\perp \text{\ for all } M\in\calM\}$ is a wide subcategory.
\item Let $\calM$ be a set of $\Lambda$-modules, then the full subcategory $\{X|X\in M^\perp \cap \leftidx{^\perp}\tau M \text{\ for all } M\in\calM\}$ is a wide subcategory.
\item If $H$ is a hereditary algebra and $\calM$ is a set of $H$-modules, then $\leftidx^{{\perp_{0,1}}}\calM$ is a wide subcategory.  
\item If $H$ is a hereditary algebra and $\calM$ is a set of $H$-modules, then $\calM^{{\perp_{0,1}}}$ is a wide subcategory. 

\end{enumerate}
\end{corollary}

\section{Acknowledgements}
The authors thank Hugh Thomas for helpful conversations.
The mathematics in this project were largely completed while the first author was a Zelevinsky Research Instructor at Northeastern University and the third author was a PhD student at Northeastern University. They wish to thank the department for a supportive research environment throughout their time there. 

\bibliography{kappa}
\bibliographystyle{plain}

\end{document}